\documentclass[12pt]{amsart}

\usepackage{amsmath,amstext,amssymb,amsopn,amsthm}
\usepackage{url,verbatim}
\usepackage{mathtools}
\usepackage{enumerate}

\usepackage{color,graphicx,subfigure}

\usepackage[margin=30mm]{geometry}
\usepackage{eucal,mathrsfs,dsfont}

\allowdisplaybreaks

\newtheorem{theorem}{Theorem}[section]
\newtheorem{corollary}[theorem]{Corollary}
\newtheorem{lemma}[theorem]{Lemma}
\newtheorem{proposition}[theorem]{Proposition}

\theoremstyle{definition}

\newtheorem{definition}[theorem]{Definition}

\newtheorem{remark}[theorem]{Remark}

\numberwithin{equation}{section}

\newcommand{\eps}{\varepsilon}

\newcommand{\calF}{\mathcal{F}}

\newcommand{\calG}{\mathcal{G}}
\newcommand{\calA}{\mathcal{A}}

\newcommand{\calT}{\mathcal{T}}
\newcommand{\calJ}{\mathcal{J}}

\newcommand{\calC}{\mathcal{C}}

\newcommand{\calM}{\mathcal{M}}

\newcommand{\calR}{\mathcal{R}}

\newcommand{\E}{\operatorname{\mathds{E}}} 
\renewcommand{\P}{\operatorname{\mathds{P}}} 
\newcommand{\R}{\mathds{R}}

\newcommand{\Z}{{\mathbb Z}}

\newcommand{\ba}{{\mathbf a}}
\newcommand{\bb}{{\mathbf b}}
\newcommand{\bc}{{\mathbf c}}
\newcommand{\bp}{{\mathbf p}}
\newcommand{\bq}{{\mathbf q}}

\newcommand{\bzero}{{\bf 0}}
\newcommand{\prt}{\partial}

\newcommand{\wh}{\widehat}
\newcommand{\wt}{\widetilde}

\DeclareMathOperator{\dist}{dist}

\DeclareMathOperator{\Var}{Var}
\DeclareMathOperator{\Cov}{Cov}

\def\bone{{\bf 1}}

\def\bt{{\bf t}}

\def\ba{{\bf a}}

\def\lra{\leftrightarrow}

\def\bt{1}

\newcommand{\bfa}{{\bf a}}
\newcommand{\bfb}{{\bf b}}
\newcommand{\bfh}{{\bf h}}
\newcommand{\bfg}{{\bf g}}

\title[Meteor process on ${\mathbb Z}^d$]{Meteor process on ${\mathbb Z}^d$}

\author{
Krzysztof Burdzy} 
\email{burdzy@math.washington.edu}
\thanks{Research partially supported by NSF grant DMS-1206276} 
\address{Department of Mathematics, Box 354350, University of Washington, Seattle, WA 98195, USA}

\begin{document}

\begin{abstract}

The meteor process is a model for mass redistribution on a graph. The case of finite graphs was analyzed in \cite{BBPS}. 
This paper is devoted to the meteor process on ${\mathbb Z}^d$. The process is constructed and a stationary distribution is found. 
Convergence to this stationary distribution is proved for a large family of initial distributions. The first two moments of the mass distribution at a vertex are computed for the stationary distribution. For the one-dimensional lattice ${\mathbb Z}$, the net flow of mass between adjacent vertices is shown to have bounded variance as time goes to infinity. An alternative representation of the process on ${\mathbb Z}$ as a collection of non-crossing paths is presented. The distributions of a ``tracer particle'' in this system of non-crossing paths are shown to be tight as time goes to infinity.

\end{abstract}

\maketitle

\section{Introduction}\label{intro}

We study a model of mass redistribution on a graph. 
A vertex $x$ of the graph holds mass $M^x_t\geq 0$ at time $t$. When a ``meteor hits''
$x$ at time $t$, the mass $M^x_t$ of the soil present at $x$ is distributed equally among all neighbors of $x$ (added to their masses). There is no soil (mass) left at $x$ just after a meteor hit. Meteor hits are modeled as independent Poisson processes, one for each vertex of the graph.
This model was studied in \cite{BBPS} in the case of finite graphs. The existence and uniqueness of the stationary distribution were proved for all connected simple graphs. The rate of convergence to the stationary distribution was estimated for some graphs. Various properties of the stationary distribution were proved. 

This paper is mostly devoted to the meteor process on $\Z^d$. 
The existence of the process is proved for arbitrary (infinite) graphs with a bounded degree in Section \ref{sec:pre}.
In Section \ref{st_ex}, a stationary distribution is found for the process on $\Z^d$, for every $d\geq 1$. In the same section, the first two moments of the mass distribution at a vertex in the stationary regime are determined.
Convergence to this stationary distribution is proved for a large family of initial distributions in Section \ref{st_conv}. In Section \ref{flow}, for the one-dimensional lattice $\Z$, the net flow of mass between adjacent vertices is shown to have bounded variance as time goes to infinity. The same section contains an alternative representation of the process on $\Z$ as a collection of non-crossing paths. The distributions of a ``tracer particle'' in this system of non-crossing paths are shown to be tight as time goes to infinity.

Section \ref{sec:supp} is the only part of the paper devoted to finite graphs. It is shown that, for finite graphs, the support of the stationary distribution is equal to the ``largest possible'' candidate for this set. 

We presented a review of related models and articles in \cite{BBPS}. The following is a shortened version of that discussion with some new references.

A model of mass redistribution similar to ours appeared in \cite{HW} but that paper went in a completely different direction. It was mostly focused on the limit model when the graph approximates the real line. 
There is a considerable literature on a mass redistribution model called ``chip-firing''; we mention here only \cite{CP, Heu}.
Mass redistribution is a part of every sandpile model, including a ``continuous'' version studied in \cite{FMQR}. See also ``divisible sandpile'' in \cite{LP}. Sandpile models have considerably different structures and associated questions from ours.
In a different direction, the reader may want to consult a paper \cite{Over} on ``overhang''. 
The introduction to \cite{BBPS} explains how our model on a \emph{finite} graph can be represented as a product of random matrices. 
This is similar to the product of random matrices that appeared
in \cite{LJL}.
So far, no technically useful connection between our model and random matrices has been found but such a connection seems to be an intriguing possibility.
 A more recent line of investigation related to our work is on Markov chains on the space of partitions---see  \cite{CL,C14}. 

An important technical tool in \cite{BBPS} and this paper is a pair of ``weakly interacting'' continuous time symmetric random walks on the graph.
They are called WIMPs for ``weakly interacting  mathematical particles.''
If the two random walks are at different vertices, they move independently. However, if they are at the same vertex, their next jumps occur at the same time, after an exponential waiting time, common to both processes. The dependence ends here---the two processes jump to vertices chosen independently, even though they jump at the same time. One can think about each of the random walks as a grain of sand. The mass present at every vertex can be thought of as a large number of very small grains of sand. WIMPs played an important role in \cite{FF}. A similar process (``associate Markov chain'') appeared in Section 2.1 of \cite{AldLan}.

\section{Construction and basic properties}\label{sec:pre}

This section contains definitions and results from \cite{BBPS}. Only Proposition \ref{de22.1} is new.

The following setup and notation will be used in most of the
paper. All constants will be assumed to be strictly positive, finite,
real numbers, unless stated otherwise.  The notation $|S|$ will be
used for the cardinality of a finite set, $S$. We will write $\bzero =(0,0, \dots, 0)$.

We will consider only connected graphs with no loops and no
multiple edges.  We will often denote the chosen graph by $G$ and its
vertex set by $V$.  In particular, we often use $k$ for $|V|$.  We let
$d_v$ stand for the degree of a vertex $v$, and write $v\lra x$ if
vertices $v$ and $x$ are connected by an edge.

We will write $\calC_k$ to denote the circular graph with $k$
vertices, $k\geq 2$. In other words, the vertex set of $\calC_k$ is
$\{1,2,\dots, k\}$ and the only pairs of vertices joined by edges are
of the form $(j, j+1)$ for $j=1,2, \dots, k-1$, and $(k,1)$. For
$\calC_k$, all arguments will apply ``mod $k$''.  For example, we will
refer to $k$ as a vertex ``to the left of 1,'' and interpret $j-1$ as
$k$ in the case when $j=1$.

Every vertex $v$ is associated with a Poisson process $N^v$ representing ``arrival times of meteors''
with intesity 1.
We assume that all processes $N^v$ are jointly independent. A vertex $v$ holds some ``soil'' with mass equal to $M^v_t\geq 0$ at time $t\geq0$. The processes $M^v$ evolve according to  the following scheme.

We assume that $M^v_0 \in [0, \infty)$ for every $v$, a.s.
At the time $t$ of a jump of $N^v$, $M^v$ jumps to 0. At the same time, the mass $M^v_{t-}$ is ``distributed'' equally among all adjacent sites, that is, for every vertex $x\lra v$,
the process $M^x$ increases by $M^v_{t-}/d_v$, that is,
$M^x_t = M^x_{t-} + M^v_{t-}/d_v$. 
The mass $M^v$ will change only when $N^v$ jumps 
and just prior to that time there is positive mass at $v$,
or $N^x$ jumps, for
some $x\lra v$ and just prior to that time there is positive mass
at $x$.  
We will denote the meteor process $\calM_t = \{M^{v}_t, v\in V\}$.  

The informal definition of the meteor process $\calM_t$ is clearly rigorous if meteor hits, i.e., jump times of processes $N^v$, do not have accumulation times. Hence, the definition does not require any more attention in the case when $V$ is finite. The case of infinite graph requires a more formal argument, presented in the following proposition.

\begin{proposition}\label{de22.1}
Suppose that $G$ is a (not necessarily finite) graph and assume that $d_G:=\sup _{v\in V} d_v < \infty$. Assume that $M^v_0 \in [0, \infty)$ for every $v$, a.s. Then there exists a unique process $\{\calM_t, t\geq 0\}$ evolving in the manner described above.
\end{proposition}

\begin{proof}
The proof is an implementation of 
the graphical construction method first proposed in \cite{Har}.
Heuristically speaking, this method works because
on short enough time intervals, we have domination by subcritical
percolation.

It will be convenient to use independent Poisson processes $N^v_t$ defined for all $t \in \R$, not only for $t\geq 0$. For a set $A\subset V$, let $U(A) = \{v \in V: \exists y \in A \text{  such that  } v \lra y\}$.

Consider any $x\in V$ and $T>0$.
Let $\Delta N^v_t = N^v_t - N^v_{t-}$.
Let $A_0=\{x\}$, $t_0 = T$ and for $j\geq 1$, let 
\begin{align*}
t_j &=\sup\{t \leq t_{j-1}: \Delta N^y_{t} \ne 0 \text{ for some  } y \in U(A_{j-1}) \},\\
y_j &= y \text {  such that  } \Delta N^y_{t_j} \ne 0,\\
A_j &= A_{j-1}
\cup \{  y_j\},\\
S_j &= t_{j-1} - t_j.
\end{align*}
We have $|A_j| \leq 1 + j $ and $|U(A_j)| \leq (1+j) d_G$.
Given $A_{j-1}$, the distribution of $S_j$ is exponential with the mean $1 / |U(A_{j-1})|$. Let $S^*_j$ be independent exponential random variables with $\E S^*_j = 1/((1+j) d_G)$. One can couple (construct on the same probability space) $S_j$'s and $S^*_j$'s so that $S^j \geq S^*_j$ for all $j\geq 1$, a.s. A straightforward application of Kolmogorov's three series theorem shows that $\sum_{j\geq 1} S^*_j = \infty$, a.s. Hence $\sum_{j\geq 1} S_j = \infty$, a.s. Let $I$ be the largest $i$ such that $\sum_{j= 1}^i S_j < T$ and note that $I < \infty$, a.s.

Recall that $x\in V$ is fixed.
We will say that a function $\{\Lambda_t, t\in [0,T]\}$ with values in $V$ is an \emph{acceptable path} if $\Lambda_T =x$,  $\Lambda$ jumps at a time $t$ if and only if $\Lambda_{t-} = v$ and $N^v$ has a jump at time $t$, and the jump takes $\Lambda$ to one of the neighbors of $v$, i.e., if $t$ is a jump time then $\Lambda_t \lra \Lambda_{t-}$.

It is easy to see that if $\Lambda$ is an acceptable path then
$\Lambda_t = x \in A_0$ for $t\in [t_1, t_0]$. 
By induction, $\Lambda_t \in A_j$ for all $t\in [t_{j+1}, t_{j}) \cap [0,T]$. Hence, $\Lambda_0 \in A_I$. It follows that the number of acceptable paths is finite, a.s.

Suppose that $\Lambda$ is an acceptable path with exactly  $j$ jumps on the interval $[0,T]$ and
let $u_1 < u_2 < \dots < u_j$ be the jump times of $\Lambda$. We will write $d(x) $ in place of $d_x$ for typographical reasons.
Let $\wt M^\Lambda_T = M^{\Lambda_0}_0 \prod_{i=1}^j
1/d(\Lambda_{u_i-})$ and let $ M^x_T = \sum_\Lambda \wt M^\Lambda_T$, where the sum is over all acceptable paths $\Lambda$. Note that $ M^x_T$ is well defined and finite, a.s.

It is easy to check that $(x,t) \to M^x_t$ has the properties described in the definition of $M^x_t$ and that it is the unique process with these properties.
\end{proof}

We will now define WIMPs (``weakly interacting mathematical particles'').

\begin{definition}\label{wimps}

Consider a finite graph.
Suppose that a meteor process $\calM$ is given and let $a= \sum_{v\in V} M^v_0$. 
For each $j\geq 1$, let $\{Y^j_n, n\geq 0\}$  be a discrete time symmetric random walk on $G$ with the initial distribution
$\P (Y^j_0 = x) =  M^x_0/a$ for $x\in V$. 
We assume that conditional on $ \calM_0$, all processes $\{Y^j_n, n\geq 0\}$, $j\geq 1$, are independent.

Recall Poisson processes $N^v$ defined earlier in this section
and assume that they are independent of $\{Y^j_n, n\geq 0\}$, $j\geq 1$.
For every $j\geq 1$, we define a continuous time Markov process $\{Z^j_t, t\geq 0\}$
by requiring that the embedded discrete Markov chain for $Z^j$ is $Y^j$ and $Z^j$ jumps at a time $t$ if and only if 
$N^v$ jumps at time $t$, where $v=Z^j_{t-}$.
Note that the jump times of all $Z^j$'s  are defined by the same family of Poisson processes $\{N^v\}_{v\in V}$.

\end{definition}

Processes $Z^j$  are continuous time nearest neighbor symmetric random walks on $G$ with exponential holding time with mean 1.
The joint distribution of $(Z^1, Z^2)$ is the following. The state space for the process $(Z^1,  Z^2)$ is $V^2$. If $(Z^1_t,  Z^2_t)=(x,y)$ with $x\ne y$ then the process will stay in this state for an exponential amount of time with mean $1/2$ and at the end of this time interval, one of the two processes (chosen uniformly) will jump to one of the nearest neighbors (also chosen uniformly). This behavior is the same as that of two independent random walks. However, if 
$(Z^1_t, Z^2_t)=(x,x)$ then the pair of processes behave in a way that is different from that of a pair of independent random walks. Namely, after an exponential waiting time with mean 1 (not $1/2$), 
both processes will jump at the same time; each one will jump to one of the nearest neighbors of $x$ chosen uniformly and independently of the direction of the jump of the other process.

\begin{remark}\label{j12.1}
The meteor process $\{\calM_t, t\geq 0\}$ is a somewhat unusual
stochastic process in that its state space can be split into an uncountable
number of disjoint communicating classes. 
For example, consider the following two initial distributions. Suppose that $M^v_0 = 1$ for all $v$. Fix some $x\in V$, and let $\wt M^v_0 = 1/\pi$ for all $v\ne x$ and $\wt M^x_0 = |V| - (|V|-1)/\pi$. If $\{\calM_t, t\geq 0\}$ and $\{\wt\calM_t, t\geq 0\}$ 
are meteor processes with these initial distributions then 
for every $t>0$, the distributions of $\calM_t$ and $\wt \calM_t$ 
will be mutually singular.

It follows from these observations that proving convergence of $\{\calM_t, t\geq 0\}$ to the stationary distribution cannot proceed along the most classical lines.  
\end{remark}

\begin{theorem}\label{a29.7}
(\cite{BBPS})
Consider the process $\{\calM_t, t\geq 0\}$ on a finite graph $G$. Assume that $|V| =k$ and $\sum_{v\in V} M^v_0 = k$. 
 When $t\to \infty$, the distribution of 
$\calM_t$
converges to a distribution $Q$ on $[0,k]^k$. The distribution $Q$ is the unique stationary distribution for the process 
$\{\calM_t, t\geq 0\}$.
In particular, $Q$ is independent of the initial
distribution of $\calM$.
\end{theorem}

\begin{remark}\label{j13.2}
It is easy to see that, for a finite graph $G$, there exists a stationary version of the process $\calM_t$ on the whole real line, i.e., there exists a process $\{\calM_t, t\in\R\}$, such that the distribution of $\calM_t$ is the stationary measure $Q$ for each $t\in \R$. Moreover, one can construct independent Poisson processes $\{N^v_t,t\in \R\} $, $v\in V$, on the same probability space, such that $\{\calM_t, t\in\R\}$ jumps according to the algorithm described above, relative to these Poisson processes. We set $N^v_0=0$ for all $v$ for definiteness.
\end{remark}

The following result has been proved in \cite{BBPS} for finite graphs.

\begin{proposition}\label{a29.8}
Suppose $G$ is finite.
Let $T^v_t$ denote the time of the last jump of $N^v$ on the interval $[0,t]$, with the convention that $T^v_t=-1$ if there were no jumps on this interval. Let $U(v)=\{v\} \cup\{x\in V: x\lra v\}$.

(i) Assume that $ M^v_0 + M^x_0 >0$ for a pair of adjacent vertices $v$ and $x$. Then, almost surely, for all $t\geq 0$,  $ M^v_t + M^x_t >0$.

(ii) 
Let $R_t$ be the number of pairs $(x,v)$ such that $x\lra v$ and 
$ M^v_t + M^x_t =0$. The process $R_t$ is non-increasing, a.s.

(iii) Assume that $  M^x_0 >0$ for  $x\in U(v)\setminus\{v\}$.
Then $M^v_t = 0$ if and only if one of the following conditions holds: (a) $T^v_t = \max \{T^x_t: x\in U(v)\}>-1$ or (b) $M^v_0 = 0$ and $ \max \{T^x_t: x\in U(v)\setminus\{v\}\}=-1$.

(iv) Suppose that the process $\{\calM_t, t\geq 0\}$ is in the stationary regime, that is, its distribution at time 0 is the stationary distribution $Q$. Then $ M^v_t + M^{x}_t >0$ for all $t\geq 0$ and all pairs of adjacent vertices $v$ and $x$, a.s.  

(v) 
Recall from Remark \ref{j13.2} the stationary meteor process $\{\calM_t, t\in\R\}$ and the corresponding Poisson processes $\{N^v_t,t\in \R\} $, $v\in V$. 
Let $T^v$ denote the time of the last jump of $N^v$ on the interval $(-\infty, 0]$ and note that $T^v$ is well defined for every $v$ because such a jump exists, a.s. 
Then $M^v_0 =0$ if an only if $T^v = \max \{T^x: x\in U(v)\}$.

\end{proposition}

\section{Stationary distribution on $\Z^d$.}\label{st_ex}

Recall that $\calC_k$ denotes the circular graph with $k$ vertices, $k\geq 2$. In other words, the vertex set of $\calC_k$ is $\{1,2,\dots, k\}$ and the only pairs of vertices joined by edges are of the form $(j, j+1)$ for $j=1,2, \dots, k-1$, and $(k,1)$.

We will show that for any $d\geq 1$, stationary distributions for meteor processes on tori $\calC_k^d$ converge, as $k\to \infty$, to a stationary distribution for the meteor process on $\Z^d$, in an appropriate sense.
We need the following notation to state the theorem.

We equip the space $\R^{\Z^d}$ with a metric $\rho$ defined by
\begin{align*}
\rho(f,g) = \sum _{x\in \Z^d}( |f(x) - g(x)|\land 1) 2^{- |x|},
\qquad f,g \in \R^{\Z^d}.
\end{align*}
Note that $\lim_{n\to \infty} \rho(f_n, g_n) = 0$ if and only if $\lim_{n\to \infty} |f_n(x) - g_n(x)|  =0$ for every $x\in \Z^d$.  We define the Skorokhod space $D([0,\infty), \R^{\Z^d})$ of RCLL functions and its topology in the usual way relative to the metric $\rho$.

For $n\geq 1$, let $K_n =\{1, 2, \dots , n\}^d\subset \Z^d$ and $K'_n = K_n - (\lfloor n/2 \rfloor, \dots, \lfloor n/2 \rfloor)$. In other words, $K'_n$ is $K_n$ is shifted so that it is (almost) centered at the origin.

Consider any $d\geq 1$.
Let $V_k$ be the vertex set of $\calC_k^d$ and let $\calM^k_t = \{M^{k,x}_t, x\in V_k\}$ be the meteor process on $\calC_k^d$, with the average mass 1 per vertex. 
Let $Q_k$ be the stationary distribution for $\calM^k$. 

Consider the graph with the vertex set $K'_k$ and edges connecting vertices at distance 1 (according to the Euclidean distance). This graph can be embedded in the obvious way into 
$\calC_k^d$. The vertex sets $K'_k$ and $V_k$ of the two graphs are in one to one correspondence so we will consider the process $\calM^k$ as a process on $K'_k$, although its transitions do not respect the edge structure of $K'_k$.

For each $k\geq 2$, $t\geq 0$ and $x\in \Z^d$, let $ M^{k,x}_t = M^{k,y}_t $, where $y\in K'_k$ is the unique vertex such that $x-y = k v$ for some $v\in \Z^d$. By abuse of notation, we will write 
$\calM^k_t = \{M^{k,x}_t, x\in \Z^d\}$ and we will use $Q_k$ to denote the stationary distribution for the process $\{M^{k,x}_t, x\in \Z^d\}$.

\begin{theorem}\label{n13.1}
(i) 
The distributions $Q_k$ converge to a distribution $Q_\infty$ on $\R^{\Z^d}$ as $k\to \infty$. 

(ii) If the distribution of $\calM^k_0$ is $Q_k$ for every $k$, then processes 
$\{\calM^{k}_t, t\geq 0\}$ converge weakly in the Skorokhod space $D([0,\infty), \R^{\Z^d})$ to a process $\{\calM^{\infty}_t, t\geq 0\}$ with the initial distribution equal to $Q_\infty$, when $k\to \infty$.
\end{theorem}

We will give the proof only in the case $d\geq 2$. The case $d=1$ can be treated using a simplified version of the argument given below.

The idea of  the proof of Theorem \ref{n13.1} is the following.
If the distributions $Q_k$ do not converge then for any sequence
of processes with these distributions constructed on the same 
probability space, there will be some pairs
of these processes with large indices which will have different 
values of the mass at some finite set of vertices, with non-vanishing
probability. This will be proved to be false by representing
masses using ``grains of sand'', that is, random walks (a single mass
can be thought of as consisting of a very large number of
grains of sand). We will couple pairs of random walks with
jump times and locations determined by the same system
of Poisson processes, and we will show that they will meet
sufficiently fast for our purposes. The precise formulation
of the coupling is given in the next lemma. Using this coupling,
we can construct a sequence of meteor processes on the same probability space,
with distributions $Q_k$,
with masses very close to each other for large indices, thus contradicting
the original assumption.

For any (discrete time or continuous time) stochastic process $R$, any set $A$ and an element $a$ of the state space of $R$ (for example, $a$ can be a number), let
 $T(R, A) = \inf\{t\geq 0: R_t \in A\}$ and
$T(R, a) = \inf\{t\geq 0: R_t =a\}$. For real $a$, we will write
$T^+(R, a) = \inf\{t\geq 0: R_t \geq a\}$ and $T^-(R, a) = \inf\{t\geq 0: R_t \leq a\}$.

\begin{lemma}\label{n14.1}
Suppose that $\{N^x_t, t\in \R\}$, $x\in \Z^d$, are independent Poisson processes.
For any $\beta\in(0,1)$ and
$\delta_1 >0$ there exist $\delta >0$, $a_0>1$, $m_1$ and $\lambda >1$ such that
\begin{align}
\label{de1.1}
(1/\beta) (1-\delta) &> 1,\\
(1+\lambda + 3 \delta)/2 &< (1/\beta) (1-\delta), \label{de1.2}
\end{align}
and for all $a\geq a_0$, $m\geq m_1$, 
and $ z_0, \wt z_0 \in \Z^d$
such that $|z _0 - \wt z_0| \leq  a^{\lambda^m }$,
one can construct a coupling  
of continuous time random walks $Z$ and $\wt Z$ on $\Z^d$, 
starting from $Z _0 = z_0$ and $\wt Z_0  = \wt z_0$,
with jumps determined by $\{N^x_t, t\in \R\}$, $x\in \Z^d$, (the same family for both $Z$ and $\wt Z$), in the sense of Section \ref{sec:pre}, and with the following properties.

(i) Let $ t_* = 2 a^{(1+\lambda+ 2\delta)\lambda^{m}}$. Then
\begin{align}\label{de4.1}
&\P\left(\sup_{0\leq t \leq t_*}|Z_{t} - Z_0| \geq a^{(1+\lambda+ 3\delta)\lambda^{m}/2} \right) \leq \delta_1,\\
&\P\left(\sup_{0\leq t \leq t_*}|\wt Z_{t} -\wt Z_0| \geq a^{(1+\lambda+ 3\delta)\lambda^{m}/2} \right) \leq \delta_1.\label{de4.2}
\end{align} 

(ii)
Let 
\begin{align}\label{de4.6}
T_* = t_* \land T^+(|Z_\cdot-  Z_0|, a^{(1+\lambda+ 3\delta)\lambda^{m}/2})
\land T^+(|\wt Z_\cdot- \wt Z_0|, a^{(1+\lambda+ 3\delta)\lambda^{m}/2} ).
\end{align}
Then 
\begin{align}\label{n14.2}
\P (T(Z- \wt Z, \bzero) < T_*  ) > 1-\delta_1.
\end{align}
\end{lemma}

The proof of the lemma is quite technical so it will
be presented at the end of Section \ref{st_ex}.

\begin{proof}[Proof of Theorem \ref{n13.1}]

{\it Step 1}. 
We will first prove part (i) of the theorem.
Since $\E_{Q_k} M^{k,x}_0 = 1$ for every $x$ and $k$, it follows that for every fixed $x$, the family of distributions of $\{M^{k,x}_0, k\geq 1\}$ is tight. Hence, for every fixed $x_1, \dots, x_j$, 
the family of $j$-dimensional distributions of $\{(M^{k,x_1}_{0}, \dots, M^{k,x_j}_{0}), k\geq 1\}$ is tight.
Using the diagonal method, we can find a subsequence $k_m$ such that 
for any $x_1, \dots, x_j$, the  distributions of $\{(M^{k_m,x_1}_{0}, \dots, M^{k_m,x_j}_{0}), m\geq 1\}$
converge. The limiting distributions are consistent by construction so there exists a distribution $Q$ on $\R^{\Z^d}$ whose restriction to any  $x_1, \dots, x_j$ is equal to the limit of the distributions of $\{(M^{k_m,x_1}_{0}, \dots, M^{k_m,x_j}_{0}), m\geq 1\}$.

Assume that part (i) of the theorem is false, i.e., $Q_k$ do not
converge to $Q$.  Then there exist $x_1, \dots, x_{i_1}\in \Z^d$ and $\ell_m$ such that  $\ell_m\to\infty$ as $m\to \infty$, and the vectors
$(M^{\ell_m,x_1}_{0}, \dots, M^{\ell_m,x_{i_1}}_{0})$ do not have the same limiting distribution as $(M^{k_m,x_1}_{0}, \dots, M^{k_m,x_{i_1}}_{0})$ when $m\to \infty$. 

Suppose that $\{\calM^{\ell_m}_{t}, t\geq 0\}$, $ m\geq 1$, and $\{\calM^{k_m}_{t}, t\geq 0\}$, $ m\geq 1$, are constructed on the same probability space. This implies that the distribution of $\calM^{\ell_m}_{t}$ is $Q_{\ell_m}$ and 
the distribution of $\calM^{k_m}_{t}$ is $Q_{k_m}$ for all $m\geq 1$ and $t\geq0$. But we do not assume anything about the relationship between the two families of processes; in particular, we do not assume that the family $\{\calM^{\ell_m}_{t}, t\geq 0\}$, $ m\geq 1$, is independent of $\{\calM^{k_m}_{t}, t\geq 0\}$, $ m\geq 1$.

The assumption that 
$(M^{\ell_m,x_1}_{0}, \dots, M^{\ell_m,x_{i_1}}_{0})$ and $(M^{k_m,x_1}_{0}, \dots, M^{k_m,x_{i_1}}_{0})$ do not have the same limiting distribution implies that there exist $c_1, p_1 >0$ such that for every $m_0$ there exists $m> m_0$ such that,
\begin{align}\label{n5.1}
\P\left( \sum_{i=1}^{i_1} |M^{k_m,x_i}_{0}-  M^{\ell_m,x_i}_{0} | > c_1\right) > p_1.
\end{align}
Note that $c_1$ and $p_1$ can  depend on the (``marginal'') distributions of $(M^{\ell_m,x_1}_{0}, \dots, M^{\ell_m,x_{i_1}}_{0})$ and $(M^{k_m,x_1}_{0}, \dots, M^{k_m,x_{i_1}}_{0})$ for $m\geq 1$, but they can be chosen so that they do not depend on the (``joint'') distributions of $(M^{\ell_m,x_1}_{0}, \dots, M^{\ell_m,x_{i_1}}_{0},M^{k_m,x_1}_{0}, \dots, M^{k_m,x_{i_1}}_{0})$
for $m\geq 1$.

Let $i_2 = 2\left\lceil \max_{1\leq i,j \leq i_1} |x_i-x_j|\right\rceil$.
Let $\Gamma_n^1 = \{x_1, \dots, x_{i_1}\}$ and let $\{\Gamma^j_n, j=1, \dots, i_3\}$ be the family of all sets of the form $\Gamma^1_n + i_2 v$ for some $v\in \Z^d$, such that $\Gamma^1_n + i_2 v \subset K'_n $. If $j\ne i$ then $\Gamma^j_n \cap \Gamma^i_n = \emptyset$.
Note that $i_3 = i_3(n) \geq \lfloor n/(2 i_2) \rfloor ^ d \geq c_2 n^d $
for $n \geq 2 i_2$. We obtain from \eqref{n5.1} that for every $m_0$ there exists $m> m_0$ such that,
\begin{align}\label{au6.2}
\E  \sum_{x\in K_n} |M^{k_m,x}_{0}-  M^{\ell_m,x}_{0} | 
&=
\E  \sum_{x\in K'_n} |M^{k_m,x}_{0}-  M^{\ell_m,x}_{0} | 
\geq
\E \sum_{j=1}^{i_3} \sum_{x\in \Gamma^j_n} |M^{k_m,x}_{0}-  M^{\ell_m,x}_{0} |
\nonumber \\ 
&= i_3 
 \sum_{x\in \Gamma^1_n} \E |M^{k_m,x}_{0}-  M^{\ell_m,x}_{0} | 
\geq i_3 c_1 p_1 \geq c_3 n^d. 
\end{align}

By stationarity, for every $t\geq 0$,
the distributions of $(M^{\ell_m,x_1}_{t}, \dots, M^{\ell_m,x_{i_1}}_{t})$ and $(M^{k_m,x_1}_{t}, \dots, M^{k_m,x_{i_1}}_{t})$
are the same as those of the vectors $(M^{\ell_m,x_1}_{0}, \dots, M^{\ell_m,x_{i_1}}_{0})$ and $(M^{k_m,x_1}_{0}, \dots, M^{k_m,x_{i_1}}_{0})$.
In view of the remark following \eqref{n5.1}, that formula applies also at time $t$. Thus \eqref{au6.2} also applies at any $t\geq 0$, i.e.,
\begin{align}
\E  \sum_{x\in K_n} |M^{k_m,x}_{t}-  M^{\ell_m,x}_{t} | 
& \geq c_3 n^d. \label{n5.5}
\end{align}

We will construct $\{\calM^k_t, t\geq 0\}$ on a common probability space in such a way  that the last inequality is false for large $n$ and hence part (i) of the theorem is true.

\medskip

{\it Step 2}.
Consider $k_0,\ell_0$ and $n$ such that for all $k\geq k_0$ and $\ell \geq \ell_0$, $K_n \subset K'_k \cap K'_\ell$.
Fix some $\beta \in \left(0 , 1\right)$. We will consider pairs of positive integers $n$ and $n_\beta $ such that $n$ is the smallest integer greater than or equal to $n_\beta ^{1/\beta}$ which is divisible by $n_\beta$. 
Let $K_n^1 = \{1, \dots, n_\beta\}^d$ and let $\{K^j_n, j=1, \dots, j_n\}$ be the family of all sets of the form $K^1_n + n_\beta v$ for some $v\in \Z^d$, such that $(K^1_n + n_\beta v) \cap K_n \ne \emptyset$. We will write $\calJ= \{1, \dots, j_n\}$.

 We have 
\begin{align}\label{de3.2}
\E_{Q_k} \left(\sum_{x\in K^j_n} M^{k,x}_0\right) = |K^j_n| = n_\beta^d.
\end{align}
Let $\prt K^j_n$ be the set of (nearest neighbor) edges
in $\Z^d$ such that exactly one of the endpoints is in $K^j_n$. We have $|\prt K^j_n| = 2d n_\beta^{d-1}$ so, by Theorem 5.1 and Remark 5.2 (v) of \cite{BBPS}, 
\begin{align*}
\lim_{k\to \infty} \Var_{Q_k} \left(\sum_{x\in K^j_n} M^{k,x}_0\right) = |\prt K^j_n|/(2d) = n_\beta^{d-1}.
\end{align*}
We will assume from now on that $k$ and $\ell$ are so large that  
\begin{align}\label{de3.3}
 \Var_{Q_k} \left(\sum_{x\in K^j_n} M^{k,x}_0\right) \leq  2 n_\beta^{d-1},
\qquad
 \Var_{Q_\ell} \left(\sum_{x\in K^j_n} M^{\ell,x}_0\right) \leq  2 n_\beta^{d-1}.
\end{align}
By \eqref{de3.2}-\eqref{de3.3} and H\"older's inequality,
\begin{align}\label{n10.1}
\E_{Q_k} \left|\sum_{x\in K^j_n} M^{k,x}_0 - \sum_{x\in K^j_n} M^{\ell,x}_0 \right|
&\leq 
\E_{Q_k} \left|\sum_{x\in K^j_n} M^{k,x}_0 - n_\beta^d \right| 
+ \E_{Q_k} \left|\sum_{x\in K^j_n} M^{\ell,x}_0 - n_\beta^d \right| \\
&\leq \sqrt{2} n_\beta^{(d-1)/2} + \sqrt{2} n_\beta^{(d-1)/2}
= 2 \sqrt{2} n_\beta^{(d-1)/2}. \nonumber
\end{align}

Fix $K^j_n$
and suppose that $\sum_{x\in K^j_n} M^{k,x}_0 \leq \sum_{x\in K^j_n} M^{\ell,x}_0$. Then let 
\begin{align}\label{n3.8}
a (k,\ell,j,n) &= \frac{\sum_{x\in K^j_n} M^{k,x}_0 }{ \sum_{x\in K^j_n} M^{\ell,x}_0} \leq 1,\\
M^{*,k,x}_0 &=  M^{k,x}_0, \qquad x\in K^j_n,\label{n3.9} \\
M^{*,\ell,x}_0 &= a (k,\ell,j,n) M^{\ell,x}_0, \qquad x\in K^j_n, \label{n3.10}
\\
\label{ap21.1}
\Lambda^j_n &=
\sum_{x\in K^j_n} M^{*,k,x}_0 = \sum_{x\in K^j_n} M^{*,\ell,x}_0.
\end{align}
If $\sum_{x\in K^j_n} M^{k,x}_0 \geq \sum_{x\in K^j_n} M^{\ell,x}_0$ then we interchange the roles of $k$ and $\ell$ in the definitions \eqref{n3.8}-\eqref{n3.10} so that \eqref{ap21.1} still holds.

We obtain from \eqref{ap21.1},
\begin{align*}
\sum_{x\in K_n} M^{*,k,x}_0 = \sum_{x\in K_n} M^{*,\ell,x}_0.
\end{align*}
It follows from \eqref{n10.1} that
\begin{align}\label{n10.2}
\E&\left(
\left(\sum_{x\in K_n} M^{k,x}_0 - \sum_{x\in K_n} M^{*,k,x}_0\right)
+
\left(\sum_{x\in K_n} M^{\ell,x}_0 - \sum_{x\in K_n} M^{*,\ell,x}_0\right)
\right)\\
&\leq 4 \sqrt{2} n_\beta^{(d-1)/2} c_4 n^{d(1-\beta)}
\leq c_5 n^{d - (d+1)\beta/2}. \nonumber
\end{align}

\medskip

{\it Step 3}. 
First we will choose values of parameters used in this step.
Recall that we have fixed a $\beta\in(0,1)$. We now fix  $\delta_1>0$ so small that $6\delta_1 < c_3/2$, where $c_3$ is the constant in \eqref{n5.5}. Then we
choose $a_0,m_1, \delta $ and $\lambda$ corresponding to $\beta$ and $\delta_1$ as in Lemma \ref{n14.1}. Consider $a\geq a_0$, $m\geq m_1$ and let $n_\beta$ be  such that $d n_\beta < a^{\lambda^m} \leq 2 d n_\beta$. Recall $c_5$ from \eqref{n10.2}.
We make $n$ and $m$  larger, if necessary, so that
\begin{align}\label{au7.1}
 &n^d - 2 d  n^{d-1}  ( (2d) ^{(1/\beta)(1-\delta)} n^{1-\delta} + n^\beta)
> |K_n| (1-\delta_1), \\
\label{au7.2}
& c_5 n^{d - (d+1)\beta/2} \leq \delta_1 n^d.
\end{align}
Assume that $k_0$ and $\ell_0$ are so large that for all $k\geq k_0$ and $\ell \geq \ell_0$, we have
$\dist(K_n, (K'_k)^c) \geq a^{(1+\lambda+ 3\delta)\lambda^{m}/2} $
and
$\dist(K_n, (K'_\ell)^c) \geq a^{(1+\lambda+ 3\delta)\lambda^{m}/2} $.

Suppose that $\{N^x_t, t\in \R\}$, $x\in \Z^d$, are independent Poisson processes.
Consider $k,\ell \geq 2$.
Recall that the stationary distribution $Q_k$ was extended from $K'_k $ (identified with $ \calC_k^d$) to $\Z^d$ in a periodic way.
Suppose that $\calM^k_0$ has the distribution $Q_k$ and similarly
we assume that $\calM^\ell_0$ has the distribution $Q_\ell$.
We do not need any assumptions about the relationship of 
$\calM^k_0$ and $\calM^\ell_0$ but, for definiteness, we assume that $\calM^k_0$, $\calM^\ell_0$ and $\{N^x_t, t\in \R\}$, $x\in \Z^d$, are independent.

Recall definitions \eqref{n3.8}-\eqref{ap21.1}. 
Let $\mu^{j,k}$ and $\mu^{j,\ell}$ be (random) probability measures on $K^j_n$ defined by
\begin{align*}
\mu^{j,k}(x) = {M^{*,k,x}_0}/{ \Lambda^j_n }, \qquad
\mu^{j,\ell}(x) = {M^{*,\ell,x}_0}/{ \Lambda^j_n },
\qquad x \in K^j_n.
\end{align*}

Let $Z_t$ and $\wt Z_t$ be a coupling of two continuous time nearest neighbor random walks constructed as in Lemma \ref{n14.1}, with the following initial distributions, 
\begin{align*}
\P(Z_0 = x ) = \mu^{j,k}(x), \qquad
\P(\wt Z_0 = x ) = \mu^{j,\ell}(x), \qquad x\in K^j_n.
\end{align*}
The joint distribution of $Z_0$ and $\wt Z_0$ is irrelevant to our argument but for the sake of definiteness we assume that these random variables are independent given $\calM^k_0$ and $\calM^\ell_0$.

We will now define processes $X$ and $\wt X$ closely related to $Z$ and $\wt Z$. 
We will consider a new family of Poisson processes.
For each $t\geq 0$ and $x\in K'_k$, let $\wh N^{k,x}_t = N^x_t $.
For each $t\geq 0$ and $x\in \Z^d$, let $\wh N^{k,x}_t =\wh N^{k,y}_t $, where $y\in K'_k$ is the unique point such that $x-y = k v$ for some $v\in \Z^d$.
Let
$\wh N^{\ell,x}_t =  N^x_t$ for $ t\geq 0$ and $x\in K'_\ell$.
For each $t\geq 0$ and $x\in \Z^d$, let $ N^{\ell,x}_t = N^{\ell,y}_t $, where $y\in K'_\ell$ and $x-y = \ell v$ for some $v\in \Z^d$.

Recall that we can identify the torus $\calC_k^d$ with  $K'_k$.
We assume that $X_0=Z_0\in K_n \subset \calC_k^d$ and $\wt X_0 =\wt Z_0\in K_n \subset \calC_\ell^d$.
We  define $X$ as a continuous time random walk on the torus $\calC_k^d$ with jump times defined by $\{\wh N^{k,x}_t, t\in \R\}$, $x\in \calC_k^d$.
Let $S^Z_j$ be the time of the $j$-th jump of $Z$ and let 
$S^X_j$ be the time of the $j$-th jump of $X$. We require that $X_{S^X_j} - X_{S^X_j-} =  Z_{S^Z_j} - Z_{S^Z_j-}$ for all $j$, where addition is in the sense of operation on the torus for $X$ and operation on $\Z^d$  for $Z$. The  formula defining the directions of the jumps of $X$ is informal but its meaning should be unambiguous to the reader. The conditions listed above define $X$ uniquely.

The definition of $\wt X$ is analogous, relative to $\calC_\ell^d$ and $\{\wh N^{\ell,x}_t, t\in \R\}$, $x\in \calC_\ell^d$.

 The definitions of $k_0$ and $\ell_0$ at the beginning of this step and the definition of $T_*$ in \eqref{de4.6} show that $X$ and $\wt X$ have the same trajectories as $Z$ and $\wt Z$, up to time $T_*$. 
Recall that $d n_\beta < a^{\lambda^m} \leq 2 d n_\beta$. Then $|X _0 - \wt X_0| \leq  a^{\lambda^m }$ 
and Lemma \ref{n14.1} implies that, 
\begin{align}\label{n10.20}
\P (T(X- \wt X, \bzero) < T_* ) > 1-\delta_1.
\end{align}

Since $a^{\lambda^{m}} \leq  2d n_\beta$, and in view of \eqref{de1.2}, 
\begin{align}\label{n3.7}
a^{(1+\lambda+ 3\delta)\lambda^{m}/2} 
&=
(a^{\lambda^{m}})^{(1+\lambda+ 3\delta)/2}
\leq  (2d n_\beta)^{(1+\lambda+ 3\delta)/2}\\
&\leq  (2d n_\beta)^{(1/\beta)(1-\delta)}
\leq (2d) ^{(1/\beta)(1-\delta)} n^{1-\delta}. \nonumber
\end{align}

Recall $\calJ$ from Step 2.
Let $\calA$ be the family of all $j \in \calJ$ such that 
$\dist(K^j_n, K_n^c) \geq a^{(1+\lambda+ 3\delta)\lambda^{m}/2}$. 
We will estimate the volume of $K^*_n := \bigcup_{j \in \calA} K^j_n$.
In view of \eqref{au7.1} and \eqref{n3.7},
\begin{align}\nonumber
|K^*_n| &= \left| \bigcup_{j \in \calA} K^j_n \right|
= |K_n| - \left| \bigcup_{j \in \calJ\setminus \calA} K^j_n \right|
\geq n^d - 2 d  n^{d-1}  ( a^{(1+\lambda+ 3\delta)\lambda^{m}/2} + n_\beta)\\
&\geq n^d - 2 d  n^{d-1}  ( (2d) ^{(1/\beta)(1-\delta)} n^{1-\delta} + n^\beta)
> |K_n| (1-\delta_1). \label{n10.3}
\end{align}

Let $\left\{(M^{\bt,k,x}_t)_{x\in \calC_k^d}, t\geq 0\right\} $ be the meteor process with the initial 
distribution defined by $ M^{\bt,k,x}_0 = M^{*,k,x}_0 $ if $x\in K_n$. For all other $x \in \calC_k^d \setminus K_n$, we let $ M^{\bt,k,x}_0 = 0$. The
jump times of $\calM^{1,k,x}$ are defined by $\{\wh N^{k,x}_t, t\in \R\}$, $x\in \calC_k^d$, in the usual way.
The process $\left\{(M^{\bt,\ell,x}_t)_{x\in \calC_\ell^d}, t\geq 0\right\} $ is defined in an analogous way relative to the family $\{\wh N^{\ell,x}_t, t\in \R\}$, $x\in \calC_\ell^d$, of Poisson processes, with the initial distribution $ M^{\bt,\ell,x}_0 = M^{*,\ell,x}_0 $ for $x\in K_n$.

Recall from Lemma \ref{n14.1} that $ t_* = 2 a^{(1+\lambda+ 2\delta)\lambda^{m}}$.
Let $\calG_0$ be the $\sigma$-field generated by $\calM^{\bt,k}_0$ and $\calM^{\bt,\ell}_0$. Let $\calF_*$ be the $\sigma$-field generated by $\calM^{\bt,k}_0$, $\calM^{\bt,\ell}_0$,
$\{\wh N^{k,x}_t, 0\leq t\leq t_*\}$, $x\in \calC_k^d$, 
and $\{\wh N^{\ell,x}_t, 0\leq t\leq t_*\}$, $x\in \calC_\ell^d$.
It is easy to see that, a.s.,
\begin{align*}
M^{\bt,k,x}_{t_*} = \sum_{j\in \calJ} \Lambda^j_n
\P_{\mu^{j,k}}(X_{t_*} =x \mid \calF_* ),
\qquad M^{\bt,\ell,x}_{t_*} = 
\sum_{j\in \calJ} \Lambda^j_n \P_{\mu^{j,\ell}}(\wt X_{t_*} =x \mid \calF_* ).
\end{align*}
This implies that, a.s., 
\begin{align*}
\sum_{x\in K_n} |M^{\bt,k,x}_{t_*} - M^{\bt,\ell,x}_{t_*}|
\leq 
\sum_{j\in \calJ} \Lambda^j_n \P(X_{t_*} \ne \wt X_{t_*} \mid \calF_*  ).
\end{align*}
By \eqref{n10.20},
\begin{align*}
\E&\sum_{x\in K_n} |M^{\bt,k,x}_{t_*} - M^{\bt,\ell,x}_{t_*}|
= \E \E\left(\sum_{x\in K_n} |M^{\bt,k,x}_{t_*} - M^{\bt,\ell,x}_{t_*}|
\mid \calF_* \right)\\
&\leq 
\E \E\left(
\sum_{j\in \calJ}\Lambda^j_n
\bone(X^{j,i}_{t_*} \ne \wt X^{j,i}_{t_*} )
\mid \calF_* \right)
=
\E \E\left(
\sum_{j\in \calJ}\Lambda^j_n
\bone(X^{j,i}_{t_*} \ne \wt X^{j,i}_{t_*} )
\mid \calG_0 \right)\\
&\leq \delta _1 \E \sum_{j\in \calJ} \Lambda^j_n  .
\end{align*}
Since
\begin{align*}
\E \Lambda^j_n = 
\E_{Q_k} \sum_{x\in K^j_n} M^{*,k,x}_0 \leq
\E_{Q_k} \sum_{x\in K^j_n} M^{k,x}_0 = |K^j_n| = n_\beta^d,
\end{align*}
it follows that 
\begin{align}\label{n10.5}
\E\sum_{x\in K_n} |M^{\bt,k,x}_{t_*} - M^{\bt,\ell,x}_{t_*}|
\leq   \delta_1 |\calJ| n_\beta^d =  \delta_1 (n/n_\beta)^d n_\beta^d
= \delta_1 n^d.
\end{align}

Let $\left\{(M^{2,k,x}_t)_{x\in \calC_k^d}, t\geq 0\right\} $ be the meteor process  with the initial 
distribution defined by $ M^{2,k,x}_0 = M^{k,x}_0 - M^{1,k,x}_0  $ if $x\in K_n$. For all other $x \in \calC_k^d \setminus K_n$, we let $ M^{2,k,x}_0 = 0$. The
jump times of $\calM^{2,k,x}$ are defined by $\{\wh N^{k,x}_t, t\in \R\}$, $x\in \calC_k^d$, in the usual way.
The process $\left\{(M^{2,\ell,x}_t)_{x\in \calC_\ell^d}, t\geq 0\right\} $ is defined in an analogous way relative to the family $\{\wh N^{\ell,x}_t, t\in \R\}$, $x\in \calC_\ell^d$, of Poisson processes.
It follows from \eqref{n10.2} and \eqref{au7.2} that
\begin{align}\label{n10.6}
&\E\sum_{x\in K_n} |M^{2,k,x}_{t_*} - M^{2,\ell,x}_{t_*}|
\leq
\E\sum_{x\in K_n} M^{2,k,x}_{t_*} 
+ \E\sum_{x\in K_n}  M^{2,\ell,x}_{t_*}\\
&\leq
\E\sum_{x\in \calC_k^d} M^{2,k,x}_{t_*} 
+ \E\sum_{x\in \calC_\ell^d}  M^{2,\ell,x}_{t_*} \nonumber \\
&=
\E\sum_{x\in \calC_k^d} M^{2,k,x}_{0} 
+ \E\sum_{x\in \calC_\ell^d}  M^{2,\ell,x}_{0} \nonumber \\
&= \E\left(
\left(\sum_{x\in K_n} M^{k,x}_{0} - \sum_{x\in K_n} M^{1,k,x}_{0}\right)
+
\left(\sum_{x\in K_n} M^{\ell,x}_{0} - \sum_{x\in K_n} M^{1,\ell,x}_{0}\right)
\right)\nonumber \\
&= \E\left(
\left(\sum_{x\in  K_n} M^{k,x}_{0} - \sum_{x\in  K_n} M^{*,k,x}_{0}\right)
+
\left(\sum_{x\in  K_n} M^{\ell,x}_{0} - \sum_{x\in  K_n} M^{*,\ell,x}_{0}\right)
\right)\nonumber \\
&\leq  c_5 n^{d - (d+1)\beta/2} \leq \delta_1 n^d.\nonumber
\end{align}

Let $\left\{(M^{3,k,x}_t)_{x\in \calC_k^d}, t\geq 0\right\} $ be the meteor process  with the initial 
distribution defined by $ M^{3,k,x}_0 = M^{k,x}_0   $ if $x\in \calC_k^d \setminus K_n$. For  $x \in  K_n$, we let $ M^{3,k,x}_0 = 0$. The
jump times of $\calM^{3,k,x}$ are defined by $\{\wh N^{k,x}_t, t\in \R\}$, $x\in \calC_k^d$, in the usual way.
The process $\left\{(M^{3,\ell,x}_t)_{x\in \calC_\ell^d}, t\geq 0\right\} $ is defined in an analogous way relative to the family $\{\wh N^{\ell,x}_t, t\in \R\}$, $x\in \calC_\ell^d$, of Poisson processes.
Note that for all $x \in K_n$ and $t\geq 0$,
\begin{align}\label{n10.7}
M^{k,x}_t = M^{1,k,x}_t + M^{2,k,x}_t + M^{3,k,x}_t,
\end{align}
and the analogous formula holds for $M^{\ell,x}_t$.
We have by \eqref{n10.3},
\begin{align}\nonumber
\E\sum_{x\in K_n \setminus K^*_n} & |M^{3,k,x}_{t_*} - M^{3,\ell,x}_{t_*}|
\leq  \E\sum_{x\in K_n \setminus K^*_n} (M^{3,k,x}_{t_*} 
+   M^{3,\ell,x}_{t_*})
\leq  \E\sum_{x\in  K_n \setminus K^*_n} (M^{k,x}_{t_*} 
+  M^{\ell,x}_{t_*})\\
&= \sum_{x\in  K_n \setminus K^*_n} (\E M^{k,x}_{t_*} 
+ \E M^{\ell,x}_{t_*})
= 2 |K_n \setminus K^*_n| < 2 \delta_1 |K_n| 
= 2 \delta_1 n^d. \label{n10.8}
\end{align}

Since 
$\E M^{3,k,x}_0 =  \bone_{K'_k \setminus K_n} (x)$,
one can easily show that,
\begin{align}\label{de5.1}
\E M^{3,k,x}_{t_*} 
&= \sum_{y \in K'_k \setminus K_n} \P(X_{t_*} =x  \mid X_0 =y)\\
&= \sum_{y \in K'_k \setminus K_n} \P(X_{t_*} =y  \mid X_0 =x)
= \P(X_{t_*} \in K'_k \setminus K_n \mid X_0 = x).\nonumber
\end{align}
Recall that the Hausdorff distance between $K^*_n$ and $\Z^d \setminus K_n$ is greater than $a^{(1+\lambda+ 3\delta)\lambda^{m}/2}$. We also recall that $X$ has the same trajectory as $Z$ up to time $T_*$. These observations, \eqref{de5.1} and Lemma \ref{n14.1} (i) imply that for $x\in K^*_n$, $\E M^{3,k,x}_{t_*}  \leq \delta_1$. For the same reason, $\E M^{3,\ell,x}_{t_*}  \leq \delta_1$. It follows that 
\begin{align}\nonumber
\E\sum_{x\in  K^*_n} |M^{3,k,x}_{t_*} - M^{3,\ell,x}_{t_*}|
&\leq  \E\sum_{x\in  K^*_n} (M^{3,k,x}_{t_*} 
+   M^{3,\ell,x}_{t_*})
= \sum_{x\in  K^*_n} (\E M^{3,k,x}_{t_*} 
+  \E M^{3,\ell,x}_{t_*})
\\
&\leq 2 \delta_1 | K^*_n|  < 2 \delta_1 |K_n| = 2 \delta_1 n^d. \label{n10.9}
\end{align}

Recall that we have chosen $\delta_1>0$ so that $6\delta_1 < c_3/2$.
In view of \eqref{n10.7}, the estimates \eqref{n10.5}, \eqref{n10.6}, \eqref{n10.8} and \eqref{n10.9} imply that, 
\begin{align*}
\E\sum_{x\in  K_n} |M^{k,x}_{t_*} - M^{\ell,x}_{t_*}|
&\leq  6  \delta_1 n^d < (c_3/2)\delta_1 n^d. 
\end{align*}
This contradicts \eqref{n5.5}  so the proof of part (i) of the theorem is complete.

\medskip

{\it Step 4}.
We will now prove part (ii) of the theorem. 
This is the second time in this paper that we will apply the method of graphical construction of \cite{Har}. Since $Q_k$ converge to $Q_\infty$, we can construct random vectors $\wt \calM^k_0$, $k\geq 1$, and $\wt \calM_0$ on the same space, such that $\wt \calM^k_0$ has distribution $Q_k$ for each $k\geq 1$, $\wt \calM_0$
has distribution $Q_\infty$, and $\wt\calM^k_0 \to \wt\calM_0$, a.s. Let $\{N^x_t, t\in \R\}$, $x\in \Z^d$, be independent Poisson processes, also independent of $\wt\calM^k_0$, $k\geq 1$, and $\wt\calM_0$. For each $k$, let $\{\wt\calM^k_t, t\geq 0\}$ be the meteor process with jumps determined by $\{N^x_t, t\in \R\}$, $x\in \Z^d$, and the initial value $\wt\calM^k_0$. Similarly,  let $\{\wt\calM_t, t\geq 0\}$ be the meteor process with jumps determined by $\{N^x_t, t\in \R\}$, $x\in \Z^d$, and the initial value $\wt\calM_0$.

Write $\wt\calM^k_t = (\wt M^{k,x}_t) _{x\in\Z^d}$ and $\wt\calM_t = (\wt M^{x}_t) _{x\in\Z^d}$.
The proof of Proposition \ref{de22.1} shows that, a.s.,
for a family of trajectories of $\{N^x_t, t\in \R\}$, $x\in \Z^d$, for $y \in \Z^d$ and $s_1>0$ there exists  a finite set $A\subset \Z^d$ such that the values of  $\wt M^{k,y}_s$, $k\geq 1$, and $\wt M^y_s$, $s\in[0,s_1]$, are uniquely determined by the values of $\wt M^{k,z}_0$, $k\geq 1$, and $\wt M^z_0$ for $z \in A$. Moreover, for each $k\geq 1$ and $s\in [0,s_1]$, the value of $\wt M^{k,y}_s$ is a continuous function of $\wt M^{k,z}_0$, $z\in A$. A similar remark applies to $\wt M^y_s$. This and the fact that $\wt M^{k,z}_0 \to \wt M^z_0$ a.s., when $k\to \infty$, for each $z\in A$, imply that the processes $\{\wt M^{k,y}_s, s\in[0,s_1]\}$ converge in the Skorokhod topology to   $\{\wt M^{y}_s, s\in[0,s_1]\}$ as $k\to\infty$. Since this holds for all $y\in \Z^d$, $s_1>0$ and almost all trajectories of $\{N^x_t, t\in \R\}$, $x\in \Z^d$, we conclude that $\{\wt\calM^k_t, t\geq 0\}$ converge in the Skorokhod topology to  $\{\wt\calM_t, t\geq 0\}$, a.s.

Recall the definition of $\{\calM^k_t, t\geq 0\}$ from the statement of the theorem.
Fix any $K'_n$, $t_1>0$ and $p_1 >0$. Then there exists $k_1$ so large that for $k\geq k_1$, with probability greater than $1-p_1$, there is no trajectory of any continuous time random walk with jumps determined  by $\{N^x_t, t\in \R\}$, $x\in \Z^d$, with starting point outside $K'_k$ and visiting $K'_n$ at some time in the interval $[0, t_1]$. This implies that we could construct $\{\wt\calM^k_t, t\geq 0\}$ and $\{\calM^k_t, t\geq 0\}$ on the same probability space so that with probability greater than $1-p_1$, 
$\wt M^{k,y}_t =  M^{k,y}_t$ for all $t\in[0,t_1]$ and $y \in K'_n$.
This and the fact that $\{\wt\calM^k_t, t\geq 0\}$ converge in the Skorokhod topology to  $\{\wt\calM_t, t\geq 0\}$, a.s., easily imply that $\{\calM^k_t, t\geq 0\}$ converge to  $\{\wt\calM_t, t\geq 0\}$ weakly in the Skorokhod topology.
Finally, note that the process that we call $\{\wt\calM_t, t\geq 0\}$ in this proof is the same as the process called $\{\calM^\infty_t, t\geq 0\}$ in the theorem.
\end{proof}

\begin{remark}\label{de23.2}
Suppose that $\{\calM_t, t\geq 0\}$ is the meteor process on $\Z^d$ with $\calM_0$ distributed as $Q_\infty$ defined in Theorem \ref{n13.1}. Then for every constant $c\in(0,\infty)$, the process $\{c\calM_t, t\geq 0\}$ is stationary. 
For later reference, we call the distribution of this process $Q^{(c)}_\infty$.

According to Theorem \ref{n14.6} below, the mean amount of mass at a vertex for the process $c\calM$ is equal to $c$. From the purely formal point of view, this shows that there are infinitely many stationary distributions for the meteor process on $\Z^d$ but this is a rather uninteresting observation.

Let $u(t,x) = \E M^x_t$ for $t\geq 0$ and $x\in \Z^d$. It is easy to see that $u(t,x)$ satisfies the heat equation on $\Z^d$. If the process $\calM$ is in the stationary regime then $u(t,x)$ does not depend on $t$ so $x\to u(t,x)$ is harmonic. A non-negative harmonic function on $\Z^d$ is constant. There exist many stationary meteor processes on $\Z^d$ with $u(t,x) \equiv 1$. One can construct them as in Remark \ref{de23.1} (i) below, as mixtures of processes with distributions $Q^{(c)}_\infty$.
It would be interesting to know whether there exist any  stationary distributions with the average mass 1 at a vertex which are not mixtures of distributions $Q^{(c)}_\infty$.

\end{remark}

\begin{theorem}\label{n14.6}
Let $Q_\infty$ be defined as in Theorem \ref{n13.1}.
Suppose that $d\geq 1$ and let
$\{\calM_t, t\geq 0\}  $ be
the meteor process  under the stationary measure $Q_\infty$.  We have
\begin{align}
&\E_{Q_\infty} M^x_0 = 1, \qquad x\in \Z^d, \label{n19.2} \\
&  \Var _{Q_\infty} M^x_0 = 1,  \qquad x\in \Z^d,  \label{n19.3}\\
& \Cov _{Q_\infty}( M^x_0, M^y_0) = -\frac 1{2d} ,\qquad x\lra y,\label{n19.4}\\
& \Cov _{Q_\infty}( M^x_0, M^y_0) = 0 , \qquad x\ne y \text{  and  } x\not\lra y.  \label{n19.5}
\end{align}
\end{theorem}

\begin{proof}
The following has been proved in Theorem 5.1 of \cite{BBPS}.
Suppose that $d\geq 1$ and let
$\{\calM_t, t\geq 0\} = \{(M^1_t,M^2_t, \dots, M^{k}_t), t\geq 0\} $ be
the meteor process on  $G=\calC_n^d$ (the product of $d$ copies of the cycle $\calC_n$), under the stationary measure $Q_k$ (here $k=n^d$). Let $V$ denote the vertex set of $\calC_n^d$ and assume that $\sum_{x\in V} M^x_0 = k$ under $Q_k$. Then 
\begin{align}
&\E_{Q_k} M^x_0 = 1, \qquad x\in V, \label{n23.2} \\
& \lim_{k\to \infty} \Var _{Q_k} M^x_0 = 1,  \qquad x\in V,  \label{n23.3}\\
&\lim_{k\to \infty} \Cov _{Q_k}( M^x_0, M^y_0) = -\frac 1{2d} ,\qquad x\lra y,\label{n23.4}\\
&\lim_{k\to \infty} \Cov _{Q_k}( M^x_0, M^y_0) = 0 , \qquad x\ne y \text{  and  } x\not\lra y.  \label{n23.5}
\end{align}

In view of Theorem \ref{n13.1} (i), formulas \eqref{n19.2}-\eqref{n19.5} follow from \eqref{n23.2}-\eqref{n23.5} provided $M^x_0$ and $(M^x_0)^2$ are uniformly integrable under $Q_k$, $k\geq 1$.

Uniform integrability of $M^x_0$ under $Q_k$, $k\geq 1$, follows from \eqref{n23.3}. It remains to prove uniform integrability of $(M^x_0)^2$ under $Q_k$, $k\geq 1$. It will suffice to show that 
\begin{align}
\limsup_{k\to \infty} \E _{Q_k} (M^x_0)^3 <\infty,  \qquad x\in V.  \label{n23.6}
\end{align}

We will define some subsets of the sets of vertices of $\calC_n^d$ and $\calC_n^{3d}$. We will suppress the dependence on $n$ in this notation.
Let $ \bzero =(0,\dots,0)\in \calC_n^d$. 
Let $\bfa$ be the  set of all vertices $(a_1, \dots , a_d)\in \calC_n^d $ such
that $|a_i| = |a_j| = 1$ for some $i\ne j$, and $a_m = 0$
for all $m\ne i,j$. Let $\bfb$ be the set of all vertices $(b_1,
\dots , b_d) \in \calC_n^d$ such that $|b_i| =2$ for some $i$, and $b_m = 0$
for all $m\ne i$. Let $\bfh$ be the set of all vertices $(h_1,
\dots , h_d)\in \calC_n^d $ such that $|h_i| =1$ for some $i$, and $h_m = 0$
for all $m\ne i$. Let $\bfg = V \setminus (\{\bzero\}\cup\bfa
\cup \bfb \cup \bfh)$.

Let $V_3$ be the set of all vertices of $\calC_n^{3d}$.
We will use the following notation for the elements of $V_3$,
\begin{align*}
x= (x^1, x^2, x^3)=((x^1_1, x^1_2, \dots, x^1_d),(x^2_1, x^2_2, \dots, x^2_d),(x^3_1, x^3_2, \dots, x^3_d)).
\end{align*}
For $y=(y_1, \dots , y_d)\in \calC_n^d$ let $\|y \|_1 = \sum_{1\leq k\leq n} |y_k|$. Let  $V_3^*$ be the set of all $x\in V_3$ such that
\begin{align*}
\max \left(\| x^1 - x^2 \|_1, \| x^2 - x^3 \|_1, \| x^3 - x^1 \|_1 \right)
\geq 5,
\end{align*}
and let $V_3^\circ$ be the interior of $V_3^*$, that is, the set of all vertices in $V_3^*$ which are not connected by an edge with a vertex in $V_3 \setminus V_3^*$.

Let $\bzero' $ be the set of all $x= (x^1, x^2, x^3)\in V_3$ such that $x^j-x^i = \bzero$ for some $1\leq i,j \leq 3$, $j\ne i$. Let $\bfa' $ be the set of all $x= (x^1, x^2, x^3)\in V_3$ such that $x^j-x^i \in \bfa$ for some $1\leq i,j \leq 3$, $j\ne i$. We define $\bfb'$ and $ \bfh'$ in an analogous manner.
Let $\bfg' = V_3 \setminus (\bzero'\cup\bfa'
\cup \bfb' \cup \bfh')$.

We will base our estimates for $\E _{Q_k} (M^x_0)^3 $  on a representation of $M^x_0$ using
WIMPs. Let  $Z^1, Z^2$ and $ Z^3$ be as in Definition \ref{wimps}.
In particular,
$\P (Z^j_0 = x) =  M^x_0/n^d$ for $j=1,2,3$ and $x\in V$.

Since the state space $\calC_n^{3d}$ for the process $(Z^1, Z^2,Z^3)$ is finite, the process has a stationary distribution. The stationary distribution is unique because all states communicate.
We will estimate the probability that $Z^1_t =  Z^2_t=Z^3_t$ under the stationary distribution. 
Let $ \{\pi_x, x\in V_3\}$ be the set of stationary probabilities for the discrete time Markov chain (the skeleton process) embedded in $(Z^1, Z^2,Z^3)$.

Let $\pi'_x = 1$ for $x\in \bzero'$,  $\pi'_x = \frac {6d-3}{8d}$
for $x\in \bfh'$ and $\pi'_x = \frac 34 $ for all other $x\in V_3$.
It was verified using computer algebra (Mathematica) that the function $\pi'$ satisfies the following equations.
\begin{align}\label{n23.10}
\pi'_x  &= \frac {1+2d}{4d} \pi'_x +  \frac 23 \pi'_y ,
\qquad x\in \bzero', y \in \bfh',\\
\pi'_x  &= \frac 1{3d} \pi'_y + \frac{2d -2}{3d} \pi'_z +  \frac 13 \pi'_x,\qquad
x\in \bfh', y \in \bfb', z\in \bfa',\label{n23.11}\\
\pi'_x  &=  \frac 1{4d^2} \pi'_v
+ \frac 2{3d}  \pi'_y + \frac{2d-2}{3d} \pi'_z +  \frac 13 \pi'_x, \qquad
x\in \bfa',v\in \bzero', y \in \bfh', z\in \bfg',\label{n23.12} \\
\pi'_x  &=  \frac 1{8d^2} \pi'_v
+ \frac 1{3d}  \pi'_y + \frac{2d-1}{3d} \pi'_z +  \frac 13 \pi'_x, \qquad
x\in \bfb',v\in \bzero', y \in \bfh', z\in \bfg', \label{n23.13}\\
\pi'_x & = \frac 23 \pi'_y +  \frac 13 \pi'_x, \qquad
x\in \bfg', y \in \bfa' \cup \bfb' \cup \bfg'.\label{n23.14}
\end{align}

We will show that the stationary distribution $\pi$ satisfies equations corresponding to \eqref{n23.10}-\eqref{n23.14} in an appropriate sense that will be made precise below. To save space, we will discuss the counterpart of only one of the above equations, namely \eqref{n23.12}. For any finite set $A$ let $\# A$ denote its cardinality. We will prove that for $x\in \bfa' \cap V_3^\circ$, the stationary distribution must satisfy 
\begin{align}\label{oct21.1}
\pi_x = &  \frac 1{4d^2} \frac{\sum_{v\in \bzero', v\lra x} \pi_v}
{ \# \{v\in \bzero', v\lra x\}}
+ \frac 2{3d}   \frac{\sum_{y\in \bfh', y\lra x} \pi_y}
{ \# \{y\in \bfh', y\lra x\}} \\
&\quad + \frac{2d-2}{3d}  \frac{\sum_{z\in \bfg', z\lra x} \pi_z}
{ \# \{z\in \bfg', z\lra x\}} 
+  \frac 13  \frac{\sum_{u\in \bfa', u\lra x} \pi_u}
{ \# \{u\in \bfa', u\lra x\}}. \nonumber
\end{align}
Consider $x\in \bfa'\cap V_3^\circ$. Suppose without loss of generality that  $x^1 - x^2 \in \bfa$. There exists a unique $v=(v^1,v^2,v^3)\in \bzero'$ which can be a state of $(Z^1, Z^2,Z^3)$ from which the process can jump to $x$. We have $v^1-v^2= \bzero$. When $v_1$ is hit by a meteor then $v^1$ will jump to $x^1$ and $v^2$ will jump to $x^2$ with probability $1/(4d^2)$. With probability $1/(4d^2)$, $v^1$ will jump to $x^2$ and $v^2$ will jump to $x^1$. The sum of the two probabilities is $2/(4d^2)$. We multiply this quantity by $1/2$ because the first meteor hit which moves the process  $(Z^1, Z^2,Z^3)$ to a new location may hit either $v^1$ (equal to $v^2$) or $v^3$, with equal probabilities. Hence, we have the factor $\frac 1{4d^2}$ in the first term on the right hand side of \eqref{oct21.1}.

There exist exactly 4 elements of $\bfh'$ which can be states of $(Z^1, Z^2,Z^3)$ from which the process can jump to $x$. 
This transition requires that the specific process, either $Z^1$ or $Z^2$, jumps, and each process has probability $1/3$ of jumping first.
The probability that the jump will go in the desirable direction is $1/(2d)$. The product of these probabilities and 4 is equal to $2/(3d)$ so this justifies the presence of the factor $\frac 2{3d} $ in the second term on the right hand side of \eqref{oct21.1}.

There exist exactly $4d-4$ elements of $\bfg'$ which can be states of $(Z^1, Z^2,Z^3)$ from which the process can jump to $x$. 
This transition requires that the specific process, either $Z^1$ or $Z^2$, jumps, and each process has probability $1/3$ of jumping first.
The probability that the jump will go in the desirable direction is $1/(2d)$. The product of these probabilities and $4d-4$ is equal to $(2d-2)/(3d)$ so this justifies the presence of the factor $\frac{2d-2}{3d} $ in the third term on the right hand side of \eqref{oct21.1}.

Finally, the process $Z^3$ will jump first with probability $1/3$. There exist exactly $2d$ elements of $\bfa'$ which can be states of $(Z^1, Z^2,Z^3)$ from which the process can jump to $x$.  The probability that $Z^3$ will jump in the desirable direction is $1/(2d)$. We have $ (1/3) (2d) (1/(2d)) = 1/3$ so this explains the factor $ \frac 13 $ in the last term on the right hand side of \eqref{oct21.1}.

If all occurrences of the function $\pi$ are replaced by $\pi'$ in \eqref{oct21.1} then this equation reduces to \eqref{n23.12}. A similar argument applies to the appropriate counterparts of other equations in the system 
\eqref{n23.10}-\eqref{n23.14} so we conclude that $\pi$ and $\pi'$ satisfy the same system of equilibrium equations on $V_3^\circ$. 

Let $q_{x,y}$ denote the one step transition probabilities for the skeleton process (discrete time Markov chain) embedded in $(Z^1, Z^2,Z^3)$.
Let $O=\{(x^1,x^2,x^3) \in V_3: x^1=x^2=x^3\}$. It is easy to see that there exist $p_1>0$ and $k_1$ (independent of $n$) such that for every 
$x\in V_3 \setminus V_3^\circ$ there exist $y \in O$ and a sequence $z_1, z_2, \dots, z_j$ such that $j \leq k_1$, $z_1 = y$, $z_j= x$, $z_m \lra z_{m+1}$ for all $m=1,\dots, j-1$, and 
\begin{align}\label{n24.1}
\prod_{1\leq m \leq j-1} q_{z_m, z_{m+1}} > p_1.
\end{align}

By symmetry, $\pi_x = \pi_y$ for all $x,y \in O$. Let $\pi_o$ denote this common value. If follows from \eqref{n24.1} that $\pi_x \geq p_1 \pi_o$ for all $x\in V_3 \setminus V_3^\circ$.
Since $\pi'_x \leq 1$ for all $x$, we obtain $\pi_x/\pi'_x \geq p_1 \pi_o$ for all $x\in V_3 \setminus V_3^\circ$. 
Let $\alpha = \min_{x\in V_3} \pi_x/\pi'_x$. If this minimum is attained in $V_3^\circ$ then it is easy to check, using the fact that $\pi$ and $\pi'$ satisfy the same equilibrium equations on $V_3^\circ$, that $\pi_x/\pi_x'$ also attains the minimum $\alpha$ on $V_3 \setminus V_3^\circ$, and, therefore, $\min_{x\in V_3} \pi_x/\pi'_x\geq p_1 \pi_o$.
Since $\pi'_x \geq (6d-3)/(8d)$ for all $x$, we obtain $\pi_x  \geq p_1 \pi_o(6d-3)/(8d)$ for $x\in V_3^*$. This, and the previously derived estimate for $x\in V_3 \setminus V_3^\circ$ show that if we let $c_1 = p_1 (6d-3)/(8d)>0$ then for all $x\in V_3$, 
\begin{align}\label{n24.2}
\pi_x  \geq p_1 \pi_o(6d-3)/(8d) = c_1 \pi_o.
\end{align}

Let $ \{\pi^*_x, x\in V_3\}$ be the set of stationary probabilities for the process $(Z^1, Z^2,Z^3)$. The holding time for $(Z^1, Z^2,Z^3)$ has mean 1 for all states in $O$. The mean is $1/2$ or $1/3$ for all other states, depending on whether any two components of $(Z^1, Z^2,Z^3)$ are equal or not. Hence, if  we take $c_2 = c_1/3>0$ then \eqref{n24.2} yields
for  $x\in V_3$, 
\begin{align*}
\pi^*_x  \geq  c_1 \pi_o/3 = c_2 \pi^*_o.
\end{align*}
Recall that $k=n^d$.
Since $\sum _x \pi^*_x = 1$, the last formula implies that
\begin{align}\label{n23.7}
&\limsup_{n\to \infty}k^2 \P_{Q_k} (Z^1= Z^2=Z^3)
\leq c_2^{-1}.
\end{align}

Let $\calG_t = \sigma(\calM_s, 0\leq s \leq t)$.
It is easy to see that, for $x\in V$ and $j=1,2,3$,
\begin{align*}
\P_{Q_k}(Z^j_0 =x \mid \calG_0)
= M^x_0/k.
\end{align*}
The random variables $Z^1_0, Z^2_0$ and $ Z^3_0$
are conditionally independent given $\calG_0$, so
\begin{align*}
\P_{Q_k}(Z^1= Z^2=Z^3=x \mid \calG_0)
= (M^x_0 /k)^3.
\end{align*}
Thus, using invariance of the process $(Z^1, Z^2,Z^3)$
under shifts of $\calC_n^d$,
\begin{align*}
\E_{Q_k} (M^x_0 )^3 &=
k^3 \E_{Q_k} \P_{Q_k}(Z^1= Z^2=Z^3=x \mid \calG_0)
=k^3\P_{Q_k}(Z^1= Z^2=Z^3=x ) \\
&= k^2\P_{Q_k}(Z^1= Z^2=Z^3 ).
\end{align*}
This and \eqref{n23.7} yield 
\begin{align*}
\limsup_{k\to \infty}
\E_{Q_k} (M^x_0 )^3 \leq c_2^{-1}.
\end{align*}
This proves \eqref{n23.6} and thus completes the proof of the theorem.
\end{proof}

\begin{remark}
The derivation of \eqref{n23.2}-\eqref{n23.5} in \cite{BBPS} is based on an explicit solution to a set of equations similar to \eqref{n23.10}-\eqref{n23.14}.
The method breaks down if one wants to generalize Theorem \ref{n14.6} to the third or higher moments because the solution to a similar set of equations needed for a similar argument does not seem to have a tractable form.
\end{remark}

\begin{corollary}\label{n19.1}
We have for all $n\geq 1$,
\begin{align}
&\E_{Q_\infty}\left( \sum_{x\in K_n} M^x_0 \right) = n^d, \nonumber\\
&  \Var _{Q_\infty} \left( \sum_{x\in K_n} M^x_0 \right) = n^{d-1} .\label{n28.1}
\end{align}
\end{corollary}

\begin{proof}
The first formula follows directly from \eqref{n19.2}.

We have from \eqref{n19.3}-\eqref{n19.5},
\begin{align*}
\Var _{Q_\infty} \left( \sum_{x\in K_n} M^x_0 \right) &=
\sum_{x\in K_n} \Var _{Q_\infty} M^x_0
+ \sum_{x,y\in K_n, x\lra y} \Cov _{Q_\infty}( M^x_0, M^y_0)\\
&= \sum_{x\in K_n} 1
- \sum_{x,y\in K_n, x\lra y} \frac 1{2d}.
\end{align*}
Note that the contribution from each edge should be counted twice in the second sum on the right hand side.
Formula \eqref{n28.1} follows by counting the numbers of vertices in $K_n$ and edges connecting neighbors in $K_n$.
\end{proof}

Lemma \ref{n14.1} is a crucial step in the proof of Theorem
\ref{n13.1}. The lemma shows that given a system of independent Poisson
processes, one can construct two random walks with jumps determined
by this family of Poisson processes (the same for both random walks),
and such that the random walks meet after a relatively short time.
The basic idea of the proof is to use the usual mirror coupling
for each coordinate separately. This would be
quite straightforward if we could couple Poisson processes determining
jump times for the two random walks. The fact that a single family of 
Poisson processes is
used causes a problem. Namely, mirror-coupled coordinates
do not meet because they typically arrive at the meeting location
at different times. The time difference can be estimated
and it turns out to be manageable but  the mirror coupling
has to be restarted and the whole procedure requires
multiple induction arguments. On the technical side, it is
worth noting that similar arguments often work on the 
``exponential'' scale, that is, one divides space and/or
time into ``boxes'' of the diameter $2^k$, $k\in\Z$. This does not seem
to work in our case and we have to work with ``doubly
exponential'' scale $2^{2^k}$. More precisely, the argument requires
that the doubly exponential scale is $a^{\lambda^m}$ for
carefully chosen values of $a$ and $\lambda$.

\begin{proof}[Proof of Lemma \ref{n14.1}]

{\it Step 1}. 
Suppose that an i.i.d. family of Poisson processes $N^x$, $x\in \Z^d$, is given. We will construct processes $X$ and $\wt X$ so that 
$X$ ($\wt X$) jumps at a time $t$ if and only if $X_{t-} = v$ (resp., $\wt X_{t-} = v$) and $N^v$ has a jump at time $t$. 

Suppose that $X$ is given and
let $\{Y_j, j\geq0\}$ be the discrete time random walk embedded in $X$.
We will define $\wt Y$, the discrete time random walk embedded in $\wt X$, below. We will write 
\begin{align*}
&X_t = (X^1_t, X^2_t, \dots, X^d_t), \qquad
\wt X_t = (\wt X^1_t, \wt X^2_t, \dots, \wt X^d_t),\\
&Y_j = (Y^1_j, Y^2_j, \dots, Y^d_j),\qquad
\wt Y_j = (\wt Y^1_j, \wt Y^2_j, \dots, \wt Y^d_j).
\end{align*}
For $i=1,\dots, d$, let
\begin{align*}
\wt Y^i_j =
\begin{cases}
- Y^i_j + X_0 + \wt X_0 & \text{for  } 0\leq j \leq \tau:= T(|Y^i_\cdot - \wt Y^i_\cdot|, \{0,1\}),  \\
Y^i_j - Y^i_\tau + \wt Y^i_\tau  & \text {for  } j > \tau .
\end{cases}
\end{align*}

Let $\bar Y^i_j =  |Y^i_j - \wt Y^i_j|$ and note that $\bar Y^i$ is a discrete time lazy symmetric random walk, with step size 2, starting at a non-negative integer and stopped at the hitting time of $\{0,1\}$. ``Lazy'' means here that $\P(\bar Y^i_{j+1}=\bar Y^i_j) =(d-1)/d$.

Fix any $\beta \in \left(0 , 1\right)$ and $\delta_1 >0$.
We will argue that one can  choose $\delta >0$, $a_0>1$ and $\lambda >1$ satisfying  conditions \eqref{n3.4}-\eqref{o29.3}, \eqref{o26.2}-\eqref{o26.6}, \eqref{n3.6}, \eqref{o26.15} and \eqref{o27.3} stated below, for all $a\geq a_0$ and $m\geq 1$. 
Note that conditions \eqref{n3.4}-\eqref{n3.5} are the same as \eqref{de1.1}-\eqref{de1.2}.

We fix $\delta >0$ and $\lambda >1$ satisfying
\begin{align}
\label{n3.4}
(1/\beta) (1-\delta) &> 1,\\
(1+\lambda + 3 \delta)/2 &< (1/\beta) (1-\delta),\label{n3.5} \\
(1+\lambda+ 4\delta)/4 &< 1/\lambda, \label{o29.2} \\
1+\lambda+ 2\delta &>  (1+\lambda + 3 \delta)/2. \label{o29.3}
\end{align}

Let
\begin{align}\label{o26.1}
k_{m,1} = \lceil a^{(1+\lambda+ 2\delta)\lambda^{m}} - a^{(1+\lambda + 3 \delta)\lambda^{m}/2} \rceil,\\
k_{m,2} = \lfloor a^{(1+\lambda+ 2\delta)\lambda^{m}} + a^{(1+\lambda + 3 \delta)\lambda^{m}/2} \rfloor. \label{o26.4}
\end{align}

There exists $1<a_1<\infty$ so large that for all $a\geq a_1$ and $m\geq 1$, we have, in view of \eqref{o29.3},
\begin{align}
k_{m,1} &> a^{(1+\lambda+ \delta)\lambda^{m}}, \label{o26.2}\\
k_{m,2} &< 2 a^{(1+\lambda+ 2\delta)\lambda^{m}}, \label{o26.3}\\
(k_{m_1} - a^{(1+\lambda+ 2\delta)\lambda^{m}})^2
&\geq (1/2) a^{(1+\lambda + 3 \delta)\lambda^{m}} , \label{o26.5}\\
(k_{m_2} - a^{(1+\lambda+ 2\delta)\lambda^{m}})^2
&\geq (1/2) a^{(1+\lambda + 3 \delta)\lambda^{m}} . \label{o26.6}
\end{align}

We can find $a_2 \geq a_1$ such that, because of \eqref{o29.2}, for $a\geq a_2$ and $m\geq 1$,
\begin{align}\label{n3.6}
 1+ a^{(1+\lambda+ 4\delta)\lambda^{m}/4} < a^{\lambda^m/\lambda}/d = a^{\lambda^{m-1}}/d .
\end{align} 

Let
\begin{align*}
A^i_1 = \left\{ \sup\{|\wt Y^i_j - \wt Y^i_k |: k_{m,1} \leq j,k \leq k_{m,2}\}
\geq a^{(1+\lambda + 4 \delta)\lambda^{m}/4} \right \}.
\end{align*}
An application of Kolmogorov's inequality shows that 
\begin{align}\label{o26.14}
\P\left( A^i_1 \right)
\leq c_1 a^{-\delta\lambda^{m}/2}.
\end{align} 

Let
\begin{align*}
A_2^i = \{T(\bar Y^i , \{0,1\}) < T^+(\bar Y^i ,  a^{\lambda^{m+1}}) \}.
\end{align*}

By gambler's ruin formula, for $m\geq 1$, 
\begin{align}\label{o25.3}
\P\left((A_2^i)^c
\mid \bar Y^i_0 \leq  a^{\lambda^m} \right)
\leq 2 \frac{a^{\lambda^m}}{a^{\lambda^{m+1}}} 
= 2\left(a ^{1-\lambda}\right)^{\lambda^ m}. 
\end{align}

Let
\begin{align*}
A_3^i =\{T(\bar Y^i , \{0,1\} \cup [ a^{\lambda^{m+1}},\infty)) 
\leq  a^{(1+\lambda+ \delta)\lambda^{m}} \}.
\end{align*}
We have the following standard estimate,
\begin{align*}
\E\left(T(\bar Y^i , \{0,1\} \cup [ a^{\lambda^{m+1}},\infty))
\mid \bar Y^i_0 \leq  a^{\lambda^m} \right)\leq c_2 a^{(1+\lambda)\lambda^{m}} .
\end{align*}
Hence, 
\begin{align}\label{o25.4}
\P\left((A_3^i)^c
\mid \bar Y^i_0 \leq  a^{\lambda^m} \right)
\leq c_2 a^{(1+\lambda)\lambda^{m}} /a^{(1+\lambda+ \delta)\lambda^{m}} 
= c_2 a^{- \delta\lambda^{m}} . 
\end{align}

Let $c_3 = c_1 + c_2 + 14$ so that for $a\geq a_2$ and $m\geq 1$,
\begin{align}\label{o26.15}
&2\left(a ^{1-\lambda}\right)^{\lambda^ m}
+ c_1 a^{- \delta\lambda^{m}/2} 
+ c_2 a^{- \delta\lambda^{m}} 
 +  14 a^{- \delta\lambda^{m}} 
\leq 2\left(a ^{1-\lambda}\right)^{\lambda^ m}
+  c_3 a^{- \delta\lambda^{m}/2}.
\end{align}

Recall that $\delta_1>0$ has been fixed. Given $\lambda, \delta$ and $a_2$ chosen so far, we can find $a_0 \geq a_2$ so that, for $a\geq a_0$, we have
\begin{align}\label{o27.3}
\sum_{m=1}^\infty d \left(2\left(a ^{1-\lambda}\right)^{\lambda^ m}
+ c_3 a^{- \delta\lambda^{m}/2} \right ) < \delta_1/4 \land 1/2.
\end{align}
This completes the specification of $\delta, \lambda$ and $a_0$.

Let $\wh X_t = \wt X_t$ for $t\leq T(X- \wt X, \bzero)$ and let $\{\wh X_t, t \geq T(X- \wt X, \bzero)\}$ be a continuous random walk on $\Z^d$ with the same skeleton process as that of $\{\wt X_t, t \geq T(X- \wt X, \bzero)\}$, but with holding times independent of $X$ and $\{\wt X_t, t \leq T(X- \wt X, \bzero)\}$.
Let $S(j)$ be the time of the $j$-th jump of $X$, let $\wt S(j)$ be the time of the $j$-th jump of $\wt X$, and let $\wh S(j)$ be the time of the $j$-th jump of $\wh X$. 

Let $t_1=a^{(1+\lambda+ 2\delta)\lambda^{m}}$. Let $M$ be the number of jumps of $X$ before $t_1$ and let $\wh M$ be the number of jumps of $\wh X$ before $t_1$.
Recall $k_{m,1}$ from \eqref{o26.1}.
The random variable $S(k_{m,1}) $ is the sum of $k_{m,1}$ independent exponential random variables. By  the Chebyshev inequality, using \eqref{o26.5}, for $m\geq 1$,
\begin{align}
\P(M \leq k_{m,1}) 
&= \P(S(k_{m,1}) \geq t_1)
\leq \frac{k_{m,1}}{(k_{m,1} - t_1)^2}
\leq 2\frac{\lceil a^{(1+\lambda+ 2\delta)\lambda^{m}} - a^{(1+\lambda + 3 \delta)/2\lambda^{m}} \rceil}{(a^{(1+\lambda + 3 \delta)\lambda^{m}/2})^2 }
\nonumber \\
&\leq 2 \frac{ a^{(1+\lambda+ 2\delta)\lambda^{m}}}{a^{(1+\lambda+ 3\delta)\lambda^{m}} } = 2 a^{-\delta\lambda^{m}}. \label{o26.7}
\end{align}
For the same reason, we have,
\begin{align}\label{o26.8}
\P(\wh M \leq k_{m,1}) \leq  2 a^{-\delta\lambda^{m}}.
\end{align}
A similar calculation using \eqref{o26.3} and \eqref{o26.6} gives
\begin{align}
\P(M \geq k_{m,2})  &\leq 4 a^{-\delta\lambda^{m}},\label{o26.9}\\
\P(\wh M \geq k_{m,2})  &\leq 4 a^{-\delta\lambda^{m}}.\label{o26.10}
\end{align}

Let
\begin{align*}
A_4 = \{k_{m,1} \leq  M \leq k_{m,2},
k_{m,1} \leq  \wh M \leq k_{m,2} \}.
\end{align*}
We combine \eqref{o26.7}, \eqref{o26.8}, \eqref{o26.9} and \eqref{o26.10} to see that
\begin{align}\label{o26.13}
\P\left( A_4\right)
\geq 1- 12 a^{-\delta\lambda^{m}}.
\end{align}

Let
\begin{align*}
A_5 = \{
M \leq a^{(1+\lambda+ \delta)\lambda^{m}} \} .
\end{align*}
It follows from \eqref{o26.2} and \eqref{o26.7} that
\begin{align}\label{o26.11}
\P(A_5) \leq 2 a^{-\delta\lambda^{m}}.
\end{align}

Suppose that $(A^i_1)^c \cap A_2^i \cap A^i_3 \cap A_4 \cap A_5^c$ holds. Then 
$X^i_{t_1} = Y^i_M $ (by the definition of $M$) and $|Y^i_M - \wt Y^i_M| \leq 1$ (because $A_2^i \cap A^i_3  \cap A_5^c$ holds).
We also have $\wh X^i_{t_1} = \wt Y^i_{\wh M} $ (by the definition of $\wh M$)
and $|\wt Y^i_{\wh M} - \wt Y^i_M| \leq a^{(1+\lambda + 4 \delta)\lambda^{m}/4}$
(because $(A^i_1)^c \cap A_4$ holds). It follows that, using condition \eqref{n3.6},
\begin{align*}
|X^i_{t_1} - \wh X^i_{t_1}| 
&\leq
 |X^i_{t_1} - \wt Y^i_M| + |\wh X^i_{t_1} - \wt Y^i_M|
= |Y^i_M - \wt Y^i_M|  + |\wt Y^i_{\wh M} - \wt Y^i_M| 
\leq 1 + a^{(1+\lambda + 4 \delta)\lambda^{m}/4} \\
&\leq a^{\lambda^{m-1}}/d.
\end{align*}
Let
\begin{align}\label{de3.1}
U_m =  t_1 \land T(X- \wt X, \bzero) =
a^{(1+\lambda+ 2\delta)\lambda^{m}} \land T(X- \wt X, \bzero).
\end{align}
Then
\begin{align*}
|X^i_{U_m} - \wt X^i_{U_m}| = |X^i_{U_m} - \wh X^i_{U_m}|  
\leq |X^i_{t_1} - \wh X^i_{t_1}| \leq a^{\lambda^{m-1}}/d. 
\end{align*}
It follows from \eqref{o26.14}, \eqref{o25.3}, \eqref{o25.4}, \eqref{o26.13}, \eqref{o26.11} and \eqref{o26.15} that 
\begin{align*}
\P((A^i_1)^c \cap A_2^i \cap A^i_3 \cap A_4 \cap A_5^c)
&\geq 1 -
2\left(a ^{1-\lambda}\right)^{\lambda^ m}
+ c_1 a^{- \delta\lambda^{m}/2} 
+ c_2 a^{- \delta\lambda^{m}} 
 +  14 a^{- \delta\lambda^{m}} \\ 
&\geq 1-
2\left(a ^{1-\lambda}\right)^{\lambda^ m} -  c_3 a^{- \delta\lambda^{m}/2}.
\end{align*}
This implies that $U_m$ has the following property, 
\begin{align*}
\P\left(|X^i_{U_m} - \wt X^i_{U_m}|  \leq a^{\lambda^{m-1}}/d
\mid |X_{0} - \wt X_{0}|  \leq a^{\lambda^{m}}
\right)
\geq 1-
2\left(a ^{1-\lambda}\right)^{\lambda^ m}-  c_3 a^{- \delta\lambda^{m}/2}. 
\end{align*}
It follows that
\begin{align}\label{o27.1}
\P\left(|X_{U_m} - \wt X_{U_m}|  \leq a^{\lambda^{m-1}}
\mid |X_{0} - \wt X_{0}|  \leq a^{\lambda^{m}}
\right)
\geq 1- d\left(
2\left(a ^{1-\lambda}\right)^{\lambda^ m} + c_3 a^{- \delta\lambda^{m}/2} \right) . 
\end{align}

\medskip

{\it Step 2}. 
Next we will argue that we can define $Z$ and $\wt Z$ so that $T(Z- \wt Z, \bzero)< \infty$, a.s., for any initial distributions of $Z_0$ and $\wt Z_0$.

Suppose that $Z^*$ and $ Z^{**}$ are independent continuous time simple random walks. 
Fix any $t_2\in(0,\infty)$. It is easy to see that there exists $p_1>0$ such that 
\begin{align}\label{o27.2}
\P (T(Z^*-  Z^{**}, \bzero) < t_2 \mid |Z^* _0 - Z^{**}_0| \leq  a^{\lambda } ) > p_1.
\end{align}

We will define two continuous time simple random walks $Z$ and $\wt Z$
starting from arbitrary deterministic points $z_0$ and $\wt z_0$.
The construction will be based on a family of intermediate processes $Z^j$ and $\wt Z^j$ defined as follows.  
Let $m\geq 1$ be the smallest integer such that  
$|z _0 - \wt z_0| \leq  a^{\lambda ^m}$. Then we let 
$Z^m$ and $\wt Z^m $ be continuous time simple random walks  defined as 
$X$ and $\wt X$
at the beginning of Step 1, with $Z^m_0 = z_0$ and $\wt Z^m_0 = \wt z_0$.
If $m>1$ then let $U_m$ be defined as in \eqref{de3.1} but relative to $Z^m$ and $\wt Z^m$. If $m=1$ then let $U_1 = T(Z^1- \wt Z^1, \bzero) \land t_2$.

We continue the construction using two-stage induction. In one of the stages, the index will decrease.
Suppose that we have defined $Z^j$, $\wt Z^j$ and $U_j$ for $j=m, m-1, m-2, \dots, n$. 
If $Z^n_{U_n} = \wt Z^n_{U_n} $
or $n=1$ or $|Z^n_{U_n} - \wt Z^n_{U_n} | > a ^{\lambda ^{n-1}}$
 then 
we stop the induction, that is, we do not define $Z^{n-1}$, $\wt Z^{n-1}$ and $U_{n-1}$.
Suppose that one of these events occurred. Let $R_1 = \sum_{n\leq k\leq m} U_k$.
We define $W^1_t$ and $\wt W^1_t$ for
$t\in [0, R_1]$ by setting $W^1_0 = z_0$, $\wt W^1_0 = \wt z_0$,
\begin{align*}
W^1_t = Z^j\left( t- \sum_{j+1\leq k\leq m} U_k\right), \qquad 
\wt W^1_t = \wt Z^j\left( t- \sum_{j+1\leq k\leq m} U_k\right) ,
\end{align*}
for 
$t \in \left(\sum_{j+1\leq k\leq m} U_k, \sum_{j\leq k\leq m} U_k \right]$
and $j = m,m-1, \dots, n$.

Next suppose that $n>1$, $Z^n_{U_n} \ne \wt Z^n_{U_n} $ and $|Z^n_{U_n} - \wt Z^n_{U_n} | \leq a ^{\lambda ^{n-1}}$.
Then we construct $Z^{n-1}$ and $\wt Z^{n-1}$ in the same way as $X$ and $\wt X$ were constructed at the beginning of Step 1, with $Z^{n-1}_0 = Z^n_{U_n}$ and $\wt Z^{n-1}_0 = \wt Z^n_{U_n}$. We require that jumps of these processes are determined by the family $\{N^x_t, t > \sum_{n\leq k\leq m} U_k\}$ in the sense that $Z^{n-1}_t$ jumps at a time $s > 0$ if and only if the Poisson process $N^x$ jumps at time $s - \sum_{n\leq k\leq m} U_k$, where $x = Z^{n-1}_{s-}$. Similarly, $\wt Z^{n-1}_t$ jumps at a time $s > 0$ if and only if the Poisson process $N^x$ jumps at time $s - \sum_{n\leq k\leq m} U_k$, where $x = \wt Z^{n-1}_{s-}$. We construct the skeleton processes $Y$ and $\wt Y$ for $Z^{n-1}$ and $\wt Z^{n-1}$ so that they start from $Y_0 = Z^n_{U_n}$ and $\wt Y_0 = \wt Z^n_{U_n}$ but otherwise they are independent of $\{Z^j_t, t\leq U_j\}$ and  $\{ \wt Z^j_t, t\leq U_j\}$ for $j=m,m-1, \dots, n$. 

Note that the inductive procedure necessarily ends because the parameter $n$ cannot decrease below 1.

By the strong Markov property, \eqref{o27.3}, \eqref{o27.1} and \eqref{o27.2}, 
\begin{align}\label{o27.4}
&\P\left(W^1_{R_1} = \wt W^1_{R_1}\right)\\
 &\geq
\P\left(\left(
\bigcap_{2 \leq j \leq m} \left\{ |Z^j_{U_j} - \wt Z^j_{U_j} | 
\leq a ^{\lambda ^{j-1}}\right\}
\cap \left\{Z^1_{U_1} = \wt Z^1_{U_1}\right\}
\right)
\cup \bigcup_{2 \leq j \leq m} \left\{Z^j_{U_j} = \wt Z^j_{U_j}\right\}
\right)  \nonumber \\
&\geq \left( 1 - \sum _{2 \leq j \leq m} d \left(
2\left(a ^{1-\lambda}\right)^{\lambda^ m} + c_3 a^{- \delta\lambda^{m}/2} \right) \right)
p_1 > p_1/2.\nonumber
\end{align}
Note that the above estimate does not depend on $m$.

We proceed with the second induction argument.
Suppose that $W^j$, $\wt W^j$ and $R_j$ have been defined for $j = 1, 2, \dots, \ell$. If $W^\ell_{R_\ell} = \wt W^\ell_{R_\ell} $ then we let $\zeta = \sum_{i=1}^\ell R_i$. We define $Z_t$ and $\wt Z_t$ for
$t\in [0, \zeta]$ by setting $Z_0 = z_0$, $\wt Z_0 = \wt z_0$, and
\begin{align*}
Z_t = W^i\left( t- \sum_{1\leq k\leq i-1} R_k\right), \qquad 
\wt Z_t = \wt W^i\left( t- \sum_{1\leq k\leq i-1} R_k\right) ,
\end{align*}
for 
$t \in \left(\sum_{1\leq k\leq i-1} R_k, \sum_{1\leq k\leq i} R_k \right]$
and $i = 1,2, \dots, \ell$.
We let $\{Z_t , t> \zeta\}$ be a continuous time random walk with $Z_{\zeta+} = Z_{\zeta}$ but otherwise independent of $\{Z_t, t \leq \zeta\}$. We require, as usual, that $Z_t$ jumps at a time $s > \zeta$ if and only if the Poisson process $N^x$ jumps at time $s$, where $x = Z_{s-}$. We also let $\wt Z_t = Z_t$ for $t > \zeta$.   

If $W^\ell_{R_\ell} \ne \wt W^\ell_{R_\ell} $
then we construct $W^{\ell+1}$ and $\wt W^{\ell+1}$ 
in the same way as $W^1$ and $\wt W^1$ were constructed, with $W^{\ell+1}_0 = W^\ell_{R_\ell}$ and $\wt W^{\ell+1}_0 = \wt W^\ell_{R_\ell}$. We require that jumps of these processes are determined by the family $\{N^x_t, t > \sum_{1\leq k\leq \ell} R_k\}$ in the sense that $W^{\ell+1}_t$ jumps at a time $s > 0$ if and only if the Poisson process $N^x$ jumps at time $s - \sum_{1\leq k\leq \ell} R_k$, where $x = W^{\ell+1}_{s-}$. Similarly, $\wt W^{\ell+1}_t$ jumps at a time $s > 0$ if and only if the Poisson process $N^x$ jumps at time $s - \sum_{1\leq k\leq \ell} R_k$, where $x = \wt W^{\ell+1}_{s-}$. 

By \eqref{o27.4}, 
\begin{align*}
\P \left(
W^j_{R_j} \ne \wt W^j_{R_j}, j=1,2, \dots, \ell
\right) \leq (1-p_1/2)^\ell.
\end{align*}
Letting $\ell\to\infty$, we conclude that $\zeta < \infty$, a.s.

\medskip

{\it Step 3}. 
Recall that an arbitrarily small $\delta_1>0$ has been fixed in Step 1 and fix some $a\geq a_0$, $\lambda>1$ and $\delta >0$ satisfying conditions \eqref{n3.4}-\eqref{o29.3}, \eqref{o26.2}-\eqref{o26.6}, \eqref{n3.6}, \eqref{o26.15} and \eqref{o27.3}.
Since the coupling time $\zeta$ for $Z$ and $\wt Z$ constructed in Step 2 is finite, a.s., we can find $r$ so large that 
\begin{align}\label{o29.4}
\P (T(Z- \wt Z, \bzero) < T^+(|Z- \wt Z|, r) \land r \mid |Z _0 - \wt Z_0| \leq  a^{\lambda } ) > 1-\delta_1/4.
\end{align}

We will strengthen the claim proved in Step 2. 
Recall $t_1=a^{(1+\lambda+ 2\delta)\lambda^{m}}$ and $U_m $ defined in \eqref{de3.1}. For fixed $\lambda$, we may make $a_0$ larger, if necessary, so that for all $a\geq a_0$ and  $m\geq 1$ we have $\sum_{j=1}^m a^{(1+\lambda+ 2\delta)\lambda^{j}} + r
\leq 2 a^{(1+\lambda+ 2\delta)\lambda^{m}} $. Let $ t_2 = 2 a^{(1+\lambda+ 2\delta)\lambda^{m}}$.
Then the same argument which was used in \eqref{o27.4} and the inequality \eqref{o27.3} yield
\begin{align*}
 \P&\left(\left(
\bigcap_{2 \leq j \leq m} \left\{ |Z^j_{U_j} - \wt Z^j_{U_j} | 
\leq a ^{\lambda ^{j-1}}\right\}
\cap \left\{U_j \leq a^{(1+\lambda+ 2\delta)\lambda^{j}}\right\}
\right)
\cup \bigcup_{2 \leq j \leq m} \left\{Z^j_{U_j} = \wt Z^j_{U_j}\right\}
\right)   \\
&\geq  1 - \sum _{2 \leq j \leq m} d \left(
2\left(a ^{1-\lambda}\right)^{\lambda^ m} + c_3 a^{- \delta\lambda^{m}/2} \right)  > 1- \delta_1/4.\nonumber
\end{align*}
Hence
\begin{align}\label{n3.1}
\P& (T^-(|Z- \wt Z|, a^\lambda) < T^+(|Z- \wt Z|, a^{\lambda^{m+1}}) 
\land 2 a^{(1+\lambda+ 2\delta)\lambda^{m}} \mid |Z _0 - \wt Z_0| \leq  a^{\lambda^m } ) \\
&>  1-\delta_1/4 .\nonumber
\end{align}
Since $Z$ and $\wt Z$ are continuous time random walks, standard estimates show that for some $c_4$ and all sufficiently large $m$,
\begin{align}\label{n3.2}
&\P\left(\sup_{0\leq t \leq t_2}|Z_{t} - Z_0| \geq a^{(1+\lambda+ 3\delta)\lambda^{m}/2} \right) \leq
c_4 a^{- \delta\lambda^{m}/2} \leq \delta_1/4,\\
&\P\left(\sup_{0\leq t \leq t_2}|\wt Z_{t} -\wt Z_0| \geq a^{(1+\lambda+ 3\delta)\lambda^{m}/2} \right) \leq
c_4 a^{- \delta\lambda^{m}/2} \leq \delta_1/4.\label{n3.3}
\end{align} 
These estimates agree with those in part (i) of the lemma.

Let
\begin{align*}
T_* = t_2 \land T^+(|Z_\cdot-  Z_0|, a^{(1+\lambda+ 3\delta)\lambda^{m}/2} )
\land T^+(|\wt Z_\cdot- \wt Z_0|, a^{(1+\lambda+ 3\delta)\lambda^{m}/2} ).
\end{align*}
The strong Markov property applied at $T^-(|Z- \wt Z|, a^\lambda)$ and \eqref{o29.4}, \eqref{n3.1}, \eqref{n3.2} and \eqref{n3.3} imply for large $m$,
\begin{align}\label{n3.11}
\P (T(Z- \wt Z, \bzero) < T_* \mid |Z _0 - \wt Z_0| \leq  a^{\lambda^m } ) > 1-\delta_1.
\end{align}
The proof of part (ii) of the lemma is complete.
\end{proof}

\section{Convergence to stationary distribution}\label{st_conv}

\begin{theorem}\label{n13.2}
Let $Q_\infty$ be defined as in Theorem \ref{n13.1}.
Suppose that the initial distribution of $\{\calM_t, t\geq 0\}$ is shift invariant, i.e., for every $y\in \Z^d$, the distributions of 
$\{M^x_0, x\in \Z^d\}$ and 
$\{M^{x+y}_0, x\in \Z^d\}$ are identical.
Assume that for some constants $c_1$ and $\alpha < 2 d$, we have
\begin{align}
\P(M^\bzero_0 \geq 0) &=1, \label{n26.1} \\
\E M^\bzero_0 &= 1, \label{n19.7}\\
\Var \left( \sum_{x\in K_n} M^x_0 \right) &\leq c_1 n^ \alpha,
\quad \text{  for all  } n\geq 1. \label{n19.8}
\end{align}
Then the distributions of $\calM_t$ converge to $Q_\infty$ as $t\to \infty$.
\end{theorem}

\begin{proof}

We will adapt the proof of Theorem \ref{n13.1}.

\medskip

{\it Step 1}. 
Let $Q_*$ denote the distribution of $\calM_0$.
Assume that  the theorem is false, i.e., the distributions of $\calM_t$ do not converge to $Q_\infty$ as $t\to \infty$.  Then there exist $x_1, \dots, x_{i_1}\in \Z^d$ and $t_m$ such that  $t_m\to\infty$ as $m\to \infty$, and the distributions of
$\{(M^{x_1}_{t_m}, \dots, M^{x_{i_1}}_{t_m}), m\geq 1\}$ do not 
converge to the restriction of $Q_\infty$ to $\{x_1, \dots, x_{i_1}\}$.
Let $\wt\calM_t$ denote the process with the initial distribution $Q_\infty$.
Suppose that $\{\calM_t, t\geq 0\}$ and $\{\wt \calM_t, t\geq 0\}$ are constructed on the same space in such a way that the joint distribution of $(\calM, \wt \calM)$ is invariant under shifts by vectors in $\Z^d$. Then there exist $c_2, p_1 >0$ such that for every $m_0$ there exists $m> m_0$ such that,
\begin{align}\label{n14.5}
\P\left( \sum_{i=1}^{i_1} |M^{x_i}_{t_m}-  \wt M^{x_i}_{t_m} | > c_2\right) > p_1.
\end{align}

Let $i_2 = \left\lceil \max_{1\leq i \leq i_1} |x_i|\right\rceil$.
Let $\Gamma_n^1 = \{x_1, \dots, x_{i_1}\}$ and let $\{\Gamma^j_n, j=1, \dots, i_3\}$ be the family of all sets of the form $\Gamma^1_n + i_2 v$ for some $v\in \Z^d$, such that $\Gamma^1_n + i_2 v \subset K'_n $. If $j\ne i$ then $\Gamma^j_n \cap \Gamma^i_n = \emptyset$.
Note that $i_3 = i_3(n) \geq \lfloor n/(2 i_2) \rfloor ^ d \geq c_3 n^d $
for $n \geq 2 i_2$. We obtain from \eqref{n14.5} that for every $m_0$ there exists $m> m_0$ such that,
\begin{align}\label{de18.1}
\E  \sum_{x\in K_n} |M^{x}_{t_m}-  \wt M^{x}_{t_m}  | 
&=
\E  \sum_{x\in K'_n} |M^{x}_{t_m}-  \wt M^{x}_{t_m}  | 
=
\E \sum_{j=1}^{i_3} \sum_{x\in \Gamma^j_n} |M^{x}_{t_m}-  \wt M^{x}_{t_m}  |  \\
&= i_3 
 \sum_{x\in \Gamma^1_n} \E |M^{x}_{t_m}-  \wt M^{x}_{t_m}  | 
\geq i_3 c_2 p_1 \geq c_4 n^d. \nonumber
\end{align}
We will show that the last inequality is false for some $n$ and large $m$ and hence the theorem is true.

\medskip

{\it Step 2}.
Fix some $\beta \in \left(0 , 1\right)$. We will consider pairs of positive integers $n$ and $n_\beta $ such that $n$ is the smallest integer greater than or equal to $n_\beta ^{1/\beta}$ which is divisible by $n_\beta$. 
Let $K_n^1 = \{1, \dots, n_\beta\}^d$ and let $\{K^j_n, j=1, \dots, j_n\}$ be the family of all sets of the form $K^1_n + n_\beta v$ for some $v\in \Z^d$, such that $(K^1_n + n_\beta v) \cap K_n \ne \emptyset$. We will write $\calJ= \{1, \dots, j_n\}$.

 We have 
\begin{align}\label{n19.11}
\E_{Q_*} \left(\sum_{x\in K^j_n} M^{x}_0\right) =
\E_{Q_\infty} \left(\sum_{x\in K^j_n} \wt M^{x}_0\right) = |K^j_n| = n_\beta^d.
\end{align}
By Corollary \ref{n19.1},
\begin{align*}
 \Var_{Q_\infty} \left(\sum_{x\in K^j_n} \wt M^{x}_0\right) = n_\beta^{d-1}.
\end{align*}
It follows from this, \eqref{n19.7}, \eqref{n19.8}, \eqref{n19.11} and H\"older's inequality that,
\begin{align}\label{n19.10}
\E \left|\sum_{x\in K^j_n} M^{x}_0 - \sum_{x\in K^j_n} \wt M^{x}_0 \right|
&\leq 
\E \left|\sum_{x\in K^j_n} M^{x}_0 - n_\beta^d \right| 
+ \E \left|\sum_{x\in K^j_n} \wt M^{x}_0 - n_\beta^d \right| \\
&\leq \sqrt{c_1} n_\beta^{\alpha/2} +  n_\beta^{(d-1)/2}. \nonumber
\end{align}

Fix $K^j_n$
and suppose that $\sum_{x\in K^j_n} M^{x}_0 \leq \sum_{x\in K^j_n}\wt M^{x}_0$. Then let 
\begin{align}\label{n19.20}
a (j,n) &= \frac{\sum_{x\in K^j_n} M^{x}_0 }{ \sum_{x\in K^j_n} \wt M^{x}_0} \leq 1,\\
M^{*,x}_0 &=  M^{x}_0, \qquad x\in K^j_n,\label{n19.21} \\
\wt M^{*,x}_0 &= a (j,n) \wt M^{x}_0, \qquad x\in K^j_n. \label{n19.22}
\end{align}
Note that  
\begin{align}\label{n19.23}
\Lambda^j_n :=
\sum_{x\in K^j_n} M^{*,x}_0 = \sum_{x\in K^j_n} \wt M^{*,x}_0.
\end{align}
If $\sum_{x\in K^j_n} M^{x}_0 > \sum_{x\in K^j_n}\wt M^{x}_0$ then we interchange the roles of processes $\calM$ and $\wt \calM$ in the definitions \eqref{n19.20}-\eqref{n19.22} so that \eqref{n19.23} still holds.

We obtain from \eqref{n19.23}
\begin{align}\label{n19.24}
\sum_{x\in K_n} M^{*,x}_0 = \sum_{x\in K_n} \wt M^{*,x}_0.
\end{align}
It follows from \eqref{n19.10} that for some $c_5, c_6$ and $\gamma > 0$,
\begin{align}\label{n19.25}
\E&\left(
\left(\sum_{x\in K_n} M^{x}_0 - \sum_{x\in K_n} M^{*,x}_0\right)
+
\left(\sum_{x\in K_n}\wt M^{x}_0 - \sum_{x\in K_n} \wt M^{*,x}_0\right)
\right)\\
&\leq (\sqrt{c_1} n_\beta^{\alpha/2} +  n_\beta^{(d-1)/2}) c_ 5 n^{d(1-\beta)}
\leq c_6 n^{d-\gamma}. \nonumber
\end{align}

\medskip

{\it Step 3}. 
We will now choose values for parameters used in this step.
Recall that we have fixed a $\beta\in(0,1)$. We now fix  $\delta_1>0$ so small that $6\delta_1 < c_4/4$, where $c_4$ is the constant in \eqref{de18.1}. Then we
choose $a_0,m_1, \delta $ and $\lambda$ corresponding to $\beta$ and $\delta_1$ as in Lemma \ref{n14.1}. Consider $a\geq a_0$, $m\geq m_1$ and let $n_\beta$ be  such that $d n_\beta < a^{\lambda^m} \leq 2 d n_\beta$. Recall $c_6$ and $\gamma$ from \eqref{n19.25}.
We make $n$ and $m$  larger, if necessary, so that
\begin{align}
 &n^d - 2 d  n^{d-1}  ( (2d) ^{(1/\beta)(1-\delta)} n^{1-\delta} + n^\beta)
> |K_n| (1-\delta_1), \nonumber\\
\label{au7.3}
& c_6 n^{d-\gamma} \leq \delta_1 n^d.
\end{align}

Suppose that $\{N^x_t, t\in \R\}$, $x\in \Z^d$, are independent Poisson processes.
We  assume that $\calM_0$, $\wt\calM_0$ and $\{N^x_t, t\in \R\}$, $x\in \Z^d$, are independent.

Recall definitions \eqref{n19.20}-\eqref{n19.22}. 
Let $\mu^{j}$ and $\wt \mu^{j}$ be the probability measures on $K^j_n$ defined by
\begin{align*}
\mu^{j}(x) = {M^{*,x}_0}/{ \Lambda^j_n },
\qquad \wt \mu^{j}(x) = {\wt M^{*,x}_0}/{ \Lambda^j_n },
\qquad x \in K^j_n.
\end{align*}

Let $Z_t$ and $\wt Z_t$ be a coupling of two continuous time nearest neighbor random walks constructed as in Lemma \ref{n14.1}, with the following initial distribution, 
\begin{align*}
\P(Z_0 = x ) = \mu^{j}(x), \qquad
\P(\wt Z_0 = x ) = \wt\mu^{j}(x), \qquad x\in K^j_n.
\end{align*}
The joint distribution of $Z_0$ and $\wt Z_0$ is irrelevant to our argument but for the sake of definiteness we assume that these random variables are independent.

Recall that $d n_\beta < a^{\lambda^m} \leq 2 d n_\beta$. Then $|Z _0 - \wt Z_0| \leq  a^{\lambda^m }$ 
and Lemma \ref{n14.1} implies that, 
\begin{align}\label{*n10.20}
\P (T(Z- \wt Z, \bzero) < T_* ) > 1-\delta_1.
\end{align}

Recall the following estimate from \eqref{n3.7}, 
\begin{align}\label{*n3.7}
a^{(1+\lambda+ 3\delta)\lambda^{m}/2} 
\leq (2d) ^{(1/\beta)(1-\delta)} n^{1-\delta}. 
\end{align}

Recall $\calJ$ from Step 2.
Let $\calA$ be the family of all $j \in \calJ$ such that 
$\dist(K^j_n, K_n^c) \geq a^{(1+\lambda+ 3\delta)\lambda^{m}/2}$. 
Let $K^*_n = \bigcup_{j \in \calA} K^j_n$.
We have shown in \eqref{n10.3} that, 
\begin{align}
|K^*_n| > |K_n| (1-\delta_1). \label{*n10.3}
\end{align}

Let $\left\{\calM^{\bt}_t, t\geq 0\right\} $ be the meteor process with the initial 
distribution defined by $ M^{\bt,x}_0 = M^{*,x}_0 $ if $x\in K_n$. For all other $x \in \Z^d \setminus K_n$, we let $ M^{\bt,x}_0 = 0$. The
jump times of $\calM^{1,x}$ are defined by $\{ N^{k,x}_t, t\in \R\}$, $x\in \Z^d$, in the usual way.
The process $\left\{\wt \calM^{\bt}_t, t\geq 0\right\} $ is defined in an analogous way, with the initial distribution $ \wt M^{\bt,x}_0 = \wt M^{*,x}_0 $ for $x\in K_n$.

Recall from Lemma \ref{n14.1} that $ t_* = 2 a^{(1+\lambda+ 2\delta)\lambda^{m}}$.
Let $\calF_*$ be the $\sigma$-field generated by $\calM^{\bt}_0$, $\wt\calM^{\bt}_0$, and
$\{ N^{k,x}_t, 0\leq t\leq t_*\}$, $x\in \Z^d$.
Let $\calG_0$ be the $\sigma$-field generated by $\calM^{\bt}_0$ and $\wt\calM^{\bt}_0$.
We have, a.s., for all $x\in \Z^d$,
\begin{align*}
M^{\bt,x}_{t_*} = 
\sum_{j\in \calJ}\Lambda^j_n \P_{\mu_j}(Z_{t_*} =x \mid \calF_*),
\qquad \wt M^{\bt,x}_{t_*} = 
\sum_{j\in \calJ} \Lambda^j_n  \P_{\wt \mu_j}(\wt Z_{t_*} =x \mid \calF_*).
\end{align*}
This implies that, a.s., 
\begin{align*}
\sum_{x\in K_n} |M^{\bt,x}_{t_*} - \wt M^{\bt,x}_{t_*}|
\leq 
\sum_{j\in \calJ} \Lambda^j_n  \P_{\wt \mu_j}(Z_{t_*} \ne\wt Z_{t_*}  \mid \calF_*).
\end{align*}
By \eqref{*n10.20},
\begin{align*}
\E&\sum_{x\in K_n} |M^{\bt,x}_{t_*} - \wt M^{\bt,x}_{t_*}|
=
\E\E\left(\sum_{x\in K_n} |M^{\bt,x}_{t_*} - \wt M^{\bt,x}_{t_*}|
\mid \calF_* \right)\\
&\leq 
\E \E\left(
\sum_{j\in \calJ}\Lambda^j_n
\bone(X_{t_*} \ne \wt X_{t_*} )
\mid \calF_* \right)
=\E \E\left(
\sum_{j\in \calJ}\Lambda^j_n
\bone(X_{t_*} \ne \wt X_{t_*} )
\mid \calG_0 \right)\\
&\leq \delta_1 \E\sum_{j\in \calJ} \Lambda^j_n .
\end{align*}

Since
\begin{align*}
\E \Lambda^j_n = 
\E_{Q_k} \sum_{x\in K^j_n} M^{*,x}_0 \leq
\E_{Q_k} \sum_{x\in K^j_n} M^{x}_0 = |K^j_n| = n_\beta^d,
\end{align*}
it follows that 
\begin{align}\label{*n10.5}
\E\sum_{x\in K_n} |M^{\bt,x}_{t_*} - \wt M^{\bt,x}_{t_*}|
\leq   \delta_1 |\calJ| n_\beta^d =  \delta_1 (n/n_\beta)^d n_\beta^d
= \delta_1 n^d.
\end{align}

Let $\left\{\calM^{2,x}_t, t\geq 0\right\} $ be the meteor process  with the initial 
distribution defined by $ M^{2,x}_0 = M^{x}_0 - M^{1,x}_0  $ if $x\in K_n$. For all other $x \in \Z^d \setminus K_n$, we let $ M^{2,x}_0 = 0$. The
jump times of $\calM^{2,x}$ are defined by $\{ N^{k,x}_t, t\in \R\}$, $x\in \Z^d$, in the usual way.
The process $\left\{\wt \calM^{2,x}_t, t\geq 0\right\} $ is defined in an analogous way.
It follows from \eqref{n19.25} and \eqref{au7.3} that
\begin{align}\label{*n10.6}
&\E\sum_{x\in K_n} |M^{2,x}_{t_*} -\wt M^{2,x}_{t_*}|
\leq
\E\sum_{x\in K_n} M^{2,x}_{t_*} 
+ \E\sum_{x\in K_n} \wt M^{2,x}_{t_*}\\
&\leq \E\left(
\left(\sum_{x\in  K_n} M^{x}_{0} - \sum_{x\in  K_n} M^{*,x}_{0}\right)
+
\left(\sum_{x\in  K_n} \wt M^{x}_{0} - \sum_{x\in  K_n} \wt M^{*,x}_{0}\right)
\right)\nonumber \\
&\leq  c_6 n^{d-\gamma} \leq \delta_1 n^d.\nonumber
\end{align}

Let $\left\{\calM^{3,x}_t, t\geq 0\right\} $ be the meteor process  with the initial 
distribution defined by $ M^{3,x}_0 = M^{x}_0   $ if $x\in \Z^d \setminus K_n$. For  $x \in  K_n$, we let $ M^{3,x}_0 = 0$. The
jump times of $\calM^{3,x}$ are defined by $\{ N^{k,x}_t, t\in \R\}$, $x\in \Z^d$, in the usual way.
The process $\left\{\wt \calM^{3,x}_t, t\geq 0\right\} $ is defined in an analogous way.
Note that for all $x \in K_n$ and $t\geq 0$,
\begin{align}\label{*n10.7}
M^{x}_t = M^{1,x}_t + M^{2,x}_t + M^{3,x}_t,
\end{align}
and the analogous formula holds for $\wt M^{x}_t$.
We have by \eqref{*n10.3},
\begin{align}\nonumber
\E\sum_{x\in K_n \setminus K^*_n} & |M^{3,x}_{t_*} -\wt M^{3,x}_{t_*}|
\leq  \E\sum_{x\in K_n \setminus K^*_n} (M^{3,x}_{t_*} 
+  \wt M^{3,x}_{t_*})
\leq  \E\sum_{x\in  K_n \setminus K^*_n} (M^{x}_{t_*} 
+ \wt M^{x}_{t_*})\\
&= \sum_{x\in  K_n \setminus K^*_n} (\E M^{x}_{t_*} 
+ \E \wt M^{x}_{t_*})
= 2 |K_n \setminus K^*_n| < 2 \delta_1 |K_n| 
= 2 \delta_1 n^d. \label{*n10.8}
\end{align}

Since 
$\E M^{3,x}_0 =  \bone_{\Z^d \setminus K_n} (x)$,
one can easily show that
\begin{align}\label{*de5.1}
\E M^{3,x}_{t_*} 
&= \sum_{y \in \Z^d \setminus K_n} \P(Z_{t_*} =x  \mid Z_0 =y)\\
&= \sum_{y \in \Z^d \setminus K_n} \P(Z_{t_*} =y  \mid Z_0 =x)
= \P(Z_{t_*} \in \Z^d \setminus K_n \mid Z_0 = x).\nonumber
\end{align}
Recall that the Hausdorff distance between $K^*_n$ and $\Z^d \setminus K_n$ is greater than $a^{(1+\lambda+ 3\delta)\lambda^{m}/2}$.  These observations, \eqref{*de5.1} and Lemma \ref{n14.1} (i) imply that for $x\in K^*_n$, $\E M^{3,x}_{t_*}  \leq \delta_1$. For the same reason, $\E \wt M^{3,x}_{t_*}  \leq \delta_1$. It follows that 
\begin{align}\nonumber
\E\sum_{x\in  K^*_n} |M^{3,x}_{t_*} - \wt M^{3,x}_{t_*}|
&\leq  \E\sum_{x\in  K^*_n} (M^{3,x}_{t_*} 
+  \wt M^{3,x}_{t_*})
= \sum_{x\in  K^*_n} (\E M^{3,x}_{t_*} 
+ \wt \E M^{3,x}_{t_*})
\\
&\leq 2 \delta_1 | K^*_n|  < 2 \delta_1 |K_n| = 2 \delta_1 n^d. \label{*n10.9}
\end{align}

Recall that we have chosen $\delta_1>0$ so that $6\delta_1 < c_4/4$.
In view of \eqref{*n10.7}, the estimates \eqref{*n10.5}, \eqref{*n10.6}, \eqref{*n10.8} and \eqref{*n10.9} imply that, for large $n$,
\begin{align*}
\E\sum_{x\in  K_n} |M^{x}_{t_*} - \wt M^{x}_{t_*}|
&\leq  6  \delta_1 n^d < (c_4/4)\delta_1 n^d. 
\end{align*}
Recall that the joint distribution of $(\calM, \wt \calM)$ is invariant under shifts by vectors in $\Z^d$. This and the last estimate imply that for every $x\in \Z^d$, 
$\E |M^{x}_{t_*} - \wt M^{x}_{t_*}| < (c_4/4)\delta_1 $. It follows that $\E (M^{x}_{t_*} - \wt M^{x}_{t_*})^+ < (c_4/4)\delta_1 $ and $\E (\wt M^{x}_{t_*} -  M^{x}_{t_*})^+ < (c_4/4)\delta_1 $ for all $x\in \Z^d$, where $a^+ = \max(a,0)$.
This easily implies that $\E (M^{x}_{t} - \wt M^{x}_{t})^+ < (c_4/4)\delta_1 $ and $\E (\wt M^{x}_{t} -  M^{x}_{t})^+ < (c_4/4)\delta_1 $ for all $x\in \Z^d$ and $t\geq t_*$. Hence, for $t\geq t_*$,
\begin{align*}
\E\sum_{x\in  K_n} |M^{x}_{t} - \wt M^{x}_{t}|
 < (c_4/2)\delta_1 n^d. 
\end{align*}
This contradicts \eqref{de18.1} and, therefore, completes the proof.
\end{proof}

\begin{corollary}\label{n23.1}
Let $Q_\infty$ be defined as in Theorem \ref{n13.1}.
Suppose that $M^x_0$, $x\in \Z^d$, are i.i.d. non-negative random variables with
$\E M^x_0 =1$.
Then the distributions of $\calM_t$ converge to $Q_\infty$ as $t\to \infty$.
\end{corollary}

\begin{proof}

Fix an arbitrarily small $\delta >0$. 
For $0< a< \infty$, let 
$ M^{x,a}_0 = M^x_0 \bone_{\{M^x_0\leq a\}}$, and let $\{\calM^a_t, t\geq 0\}$ be the meteor process with the initial distribution $\{M^{x,a}_0, x\in \Z^d\}$.
Let $\wt M^{x,a}_0 = M^x_0 \bone_{\{M^x_0> a\}}$ and let $\{\wt \calM^a_t, t\geq 0\}$ be the meteor process with the initial distribution $\{\wt M^{x,a}_0, x\in \Z^d\}$. Suppose that $a$ is so large that $\mu(a) :=\E M^{\bzero,a}_0 > 1-\delta$
and $\E \wt M^{\bzero,a}_0 < \delta$. Since $M^{x,a}_0$, $x\in \Z^d$, are i.i.d. and bounded, we have for some finite $c_1$,
\begin{align*}
\P(M^{\bzero,a}_0 \geq 0) &=1,  \\
\E M^{\bzero,a}_0 &= \mu(a) \in(1-\delta, 1), \\
\Var \left( \sum_{x\in K_n} M^{x,a}_0 \right) &\leq c_1 n^ d,
\quad \text{  for all  } n\geq 1. 
\end{align*}
Comparing these formulas to \eqref{n26.1}-\eqref{n19.8},
we see that Theorem \ref{n13.2} implies that the distributions
of $\calM^a_t$ converge to $ Q^{\mu(a)}_\infty$ as $t\to \infty$, where
$ Q^{\mu(a)}_\infty$ is as in Remark \ref{de23.2}.
We have $\E \wt M^{x,a}_t < \delta$ for all $t >0$ and $x\in \Z^d$
by the conservation of mass and shift invariance. Since $\calM_t =  \calM^a_t + \wt \calM^a_t$ and $\delta $ is arbitrarily small, the last two claims easily imply the corollary.
\end{proof}

\begin{remark}\label{de23.1}
(i)
The condition $\alpha < 2 d$ in Theorem \ref{n13.2} cannot be relaxed. To see this, consider the following initial distribution of the process $\calM$. With probability $1/2$, $M^x_0 = 0$ for all $x\in \Z^d$. With probability $1/2$, $M^x_0 = 2$ for all $x\in \Z^d$. It is elementary to check that this distribution is shift invariant and satisfies \eqref{n26.1}-\eqref{n19.8} with $\alpha=2d$. 
Recall distributions $ Q^{c}_\infty$ from Remark \ref{de23.2}.
It follows easily from Theorem \ref{n13.2} that the distributions of $\calM_t$ converge, as $t\to \infty$, to $(1/2)Q^{0}_\infty + (1/2)Q^{2}_\infty \ne Q_\infty$.

(ii) We conjecture that Theorem \ref{n13.2} remains true even if we drop the assumption that $\calM_0$ has shift invariant distribution.
\end{remark}

\section{Flows and reflected paths}\label{flow}

We will prove a theorem about the flow of mass between adjacent sites in $\Z$. We will write $F^x_t$ to denote the net flow between $x$ and $x+1$ on the time interval $[0,t]$, for $x\in \Z$ and $t\geq 0$. More formally, 
\begin{align*}
F^x_t = \frac 12 \sum_{0\leq s\leq t}  \left((N^x_s - N^x_{s-}) M^x_{s-}
- (N^{x+1}_s - N^{x+1}_{s-}) M^{x+1}_{s-}\right) .
\end{align*}

\begin{theorem}\label{n30.2}
Consider the meteor process $\calM_t$ on $\Z$ in the stationary regime, i.e., suppose that the distribution of $\calM_0$ is $Q_\infty$. Then for every $t\geq 0$, 
\begin{align*}
\Var F^\bzero_t \leq 2.
\end{align*}
\end{theorem}

\begin{proof}
Fix any $t\geq 0$ and consider an odd integer $x> 6$. We will eventually let $x\to \infty$ so $x$ should be thought of as a large integer. Note that
\begin{align*}
F^\bzero_t - F^x_t = \sum_{1\leq y \leq x} M^y_t - \sum_{1\leq y \leq x} M^y_0.
\end{align*}
By stationarity of the process $\calM$, the distribution of $\sum_{1\leq y \leq x} M^y_s$ does not depend on $s$, so
\begin{align}\label{n29.3}
\E\left(F^\bzero_t - F^x_t\right) = 
\E\left(\sum_{1\leq y \leq x} M^y_t\right )
- \E\left(\sum_{1\leq y \leq x} M^y_0\right)=0.
\end{align}
We obtain from \eqref{n28.1},
\begin{align}\label{n29.2}
&\Var \left(F^\bzero_t - F^x_t\right) =
\Var\left( \sum_{1\leq y \leq x} M^y_t - \sum_{1\leq y \leq x} M^y_0\right)\\
&\leq  \Var\left( \sum_{1\leq y \leq x} M^y_t \right)
+2 \left( \Var\left( \sum_{1\leq y \leq x} M^y_t \right)
\Var\left( \sum_{1\leq y \leq x} M^y_0\right)\right)^{1/2}
+  \Var\left( \sum_{1\leq y \leq x} M^y_0\right)\nonumber\\
&\leq 4 . \nonumber
\end{align}

Given $\{N^y, y\in \Z\}$, $y_1,y_2\in \Z$ and $t_1 < t_2$, we will say that there is a path between $(y_1, t_1)$ and $(y_2, t_2)$ if some mass could pass from the first point to the other according to the rules of the meteor process evolution (see the definition of ``acceptable path'' in the proof of Proposition \ref{de22.1}). 
Let $A^-_x $ be the event that there is no path from any point $((x-3)/2, s)$, $s\in[0,t]$, to $(\bzero, t)$.
Let $A^+_x $ be the event that there is no path from any point $((x+3)/2, s)$, $s\in[0,t]$, to $(x, t)$.
Let $A_x = A^-_x \cap A^+_x$. It is easy to see that  $\P_{Q_\infty}(A_x) \to 1$ as $x\to \infty$. 
We obtain using \eqref{n29.3} and \eqref{n29.2},
\begin{align}\label{n29.5}
\Var\left(\left(F^\bzero_t - F^x_t\right) \bone_{A_x}\right)
& \leq \E\left(\left(F^\bzero_t - F^x_t\right) \bone_{A_x}
\right)^2\leq \E \left(F^\bzero_t - F^x_t\right)^2 \\
& = \Var \left(F^\bzero_t - F^x_t\right) \leq 4.\nonumber
\end{align}

If a random variable $\xi$ has finite variance 
or $\xi \bone_{A_x}$ has finite variance then
\begin{align}\label{n29.4}
\Var&( \xi \mid A_x) \\
&= \E((\xi - \E(\xi \mid A_x))^2 \mid A_x)\nonumber\\
&= \E\left(\left(\xi - \E(\xi \bone_{A_x}) \frac 1 {\P(A_x)}\right)^2 \bone_{A_x}\right) \frac 1 {\P(A_x)}\nonumber \\
&= \frac 1 {\P(A_x)}
\E \left( \xi^2 \bone_{A_x} - 2 \xi  \E(\xi \bone_{A_x}) \frac {\bone_{A_x}} {\P(A_x)} + \frac 1 {\P(A_x)^2} \left(\E(\xi \bone_{A_x})\right)^2 \bone_{A_x}
\right) \nonumber\\
&= \frac 1 {\P(A_x)} \left( \E(\xi \bone_{A_x})^2 - \left(\E(\xi \bone_{A_x})\right)^2 \right)
+ \frac{  \P(A_x) -1}{ \P(A_x)^2} \left(\E(\xi \bone_{A_x})\right)^2
\label{n29.7}\\
&= \frac 1 {\P(A_x)} \Var(\xi \bone_{A_x})
+ \frac{  \P(A_x) -1}{ \P(A_x)^2} \left(\E(\xi \bone_{A_x})\right)^2
\nonumber \\
&\leq \frac 1 {\P(A_x)} \Var(\xi \bone_{A_x})
\label{n29.6}.
\end{align}
Consider an arbitrary $\delta\in(0,1/2)$ and
fix $x_1$ so large that for $x\geq x_1$ we have  $\P(A_x) \geq 1-\delta$. 
We apply   \eqref{n29.4}-\eqref{n29.6} to $\xi = F^\bzero_t - F^x_t$ and we use \eqref{n29.5} to see that for $x\geq x_1$,
\begin{align}\label{o10.1}
\Var( F^\bzero_t - F^x_t \mid A_x) 
\leq  4/(1-\delta) .
\end{align}

We will now show that $F^\bzero_t $ and $F^x_t$ are conditionally uncorrelated given $A_x$.
Let $\calG_t =\sigma\{N^y_s, s\in[0,t], y \in \Z\}$. It is easy to see that there exist random
variables 
$\alpha(v,z)\geq 0$, $v,z \in \Z$, which are measurable with respect to $\calG_t$ and such that, a.s.,
for all $z \in \Z$, we have
$F^z_t = \sum_{v\in \Z} \alpha(v,z) M^v_0$.
The random variables $\alpha(v,z)$ encode the transport of the mass
from $v$ to $z$ and then to $z+1$, and from $v$
to $z+1$ and then to $z$,
along the paths ``opened'' by $N$'s. 

For $y\in \Z$, let $G_y$ be the event that there was no meteor hit at $y$ between times 0 and $t$.  We have $\P(G_y) = e^{-t}$ for all $y$. 
Suppose that $v, z \in \Z$ and $v< z$.
A part of the mass that was present at $v $ at time 0 could have moved between vertices $z$ and $z+1$ during the time interval $[0,t]$ only if the event $C(v,z):=\bigcup_{v\leq w \leq z-1} G_w $ did not occur.
By the independence of $G_y$'s, for $v<z$,
$\P\left(C(v,z)^c\right) = (1- e^{-t})^ {z-v}$.
Hence, $\P(\alpha(v,z) \ne 0) \leq (1- e^{-t})^ {z-v}$.
A similar argument yields 
$\P(\alpha(v,z) \ne 0) \leq (1- e^{-t})^ {v-z-1}$ for $v>z+1$.
Note that $\alpha(v,z) \leq 1$ for all $v$ and $z$.
These observations and \eqref{n19.3} imply that, for all $z_1, z_2 \in \Z$,
\begin{align*}
\E & \sum_{v\in \Z} \sum_{w\in \Z}| \alpha(v,z_1) M^v_0\alpha(w,z_2) M^w_0| =
\sum_{v\in \Z} \sum_{w\in \Z} \E| \alpha(v,z_1) \alpha(w,z_2)| \E| M^v_0 M^w_0|\\
&\leq \sum_{v\in \Z} \sum_{w\in \Z} (\E \alpha(v,z_1)^2)^{1/2} (\E \alpha(w,z_2)^2)^{1/2} 
(\E( M^v_0)^2)^{1/2} (\E( M^w_0)^2)^{1/2}\\
& \leq  \sum_{v\in \Z} (1- e^{-t})^ {|z_1-v|-1} < \infty.
\end{align*}
The above bound allows us to change the order of summation in the
calculation of $\Cov(F^\bzero_t ,F^x_t \mid A_x)$ below. Recall that $\bone_{A_x} = \bone_{A^-_x}\bone_{A^+_x}$. 
Since
\begin{align*}
A^-_x \in \sigma\{N^y_s, s\in[0,t], y \leq (x-3)/2\}
\text{  and  }
A^+_x \in \sigma\{N^y_s, s\in[0,t], y \geq (x+3)/2\},
\end{align*}
the events $A^-_x $  and $A^+_x $ are independent.
This implies independence of the following pairs of random variables for $v\leq (x-3)/2$ and $w\geq (x+3)/2$: $ \alpha(v,\bzero) M^v_0   \bone_{A^-_x}$  and $\bone_{A^+_x}$; 
 $\alpha(w,x) M^w_0   \bone_{A^-_x}$ and $\bone_{A^+_x}$. All these remarks imply that
\begin{align*}
&\Cov(F^\bzero_t ,F^x_t \mid A_x)
= \E(F^\bzero_t F^x_t\mid A_x)
- \E(F^\bzero_t \mid A_x) \E(F^x_t\mid A_x)\\
& = \E \left (\sum_{v\in \Z} \alpha(v,\bzero) M^v_0  \sum_{w\in \Z} \alpha(w,x) M^w_0 \bone_{A_x} \right)
/\P(A_x) \\
& \quad - \left[\E \left (\sum_{v\in \Z} \alpha(v,\bzero) M^v_0   \bone_{A_x} \right)
/\P(A_x)\right] \cdot
\left[\E \left (\sum_{w\in \Z} \alpha(w,x) M^w_0   \bone_{A_x} \right)
/\P(A_x)\right] \\
& = \E \left (\sum_{v\leq (x-3)/2} \alpha(v,\bzero) M^v_0 \bone_{A^-_x} \sum_{w\geq (x+3)/2} \alpha(w,x) M^w_0 \bone_{A^+_x} \right)
/\P(A_x) \\
& \quad - \left[\E \left (\sum_{v\leq (x-3)/2} \alpha(v,\bzero) M^v_0   \bone_{A^-_x}\bone_{A^+_x} \right)
/\P(A_x)\right] \times \\
& \qquad \times \left[\E \left (\sum_{w\geq (x+3)/2} \alpha(w,x) M^w_0   \bone_{A^-_x}\bone_{A^+_x} \right)
/\P(A_x)\right]\\
& = \sum_{v\leq (x-3)/2} \sum_{w\geq (x+3)/2} \Big[
\E\left( \alpha(v,\bzero) M^v_0 \bone_{A^-_x}  \alpha(w,x) M^w_0 \bone_{A^+_x}\right )
/\P(A_x) \\
& \quad - \E \left ( \alpha(v,\bzero) M^v_0   \bone_{A^-_x} \right)\left( \P(A^+_x)
/\P(A_x)\right) \E \left ( \alpha(w,x) M^w_0  \bone_{A^+_x} \right) \left( \P(A^-_x)
/\P(A_x)\right) \Big]\\
& = \left( 1/\P(A_x)\right)\sum_{v\leq (x-3)/2} \sum_{w\geq (x+3)/2} 
 \Big[
\E\left( \alpha(v,\bzero) M^v_0 \bone_{A^-_x}  \alpha(w,x) M^w_0 \bone_{A^+_x}\right ) \\
&\qquad\qquad\qquad - \E \left ( \alpha(v,\bzero) M^v_0   \bone_{A^-_x} \right)
 \E \left ( \alpha(w,x) M^w_0  \bone_{A^+_x} \right) \Big] .
\end{align*}
Each term in the last sum is equal to 0 because for any $v\leq (x-3)/2$ and $w\geq (x+3)/2$,
the random variables $M^v_0 $ and $ M^w_0$ are uncorrelated (see \eqref{n19.5}), the random variables
$\alpha(v,\bzero)  \bone_{A^-_x}$ and $\alpha(w,x) M^w_0 \bone_{A^+_x}$ are independent, and so are the random 
variables $\alpha(v,\bzero) M^v_0 \bone_{A^-_x} $ and $ \alpha(w,x)  \bone_{A^+_x}$.
We conclude that $\Cov(F^\bzero_t ,F^x_t \mid A_x)=0$, i.e., $F^\bzero_t $ and $F^x_t$ are conditionally uncorrelated given $A_x$. This and \eqref{o10.1} imply that
\begin{align*}
\Var( F^\bzero_t \mid A_x) + \Var( F^x_t \mid A_x) = \Var( F^\bzero_t - F^x_t \mid A_x) 
\leq  4/(1-\delta) .
\end{align*}
By symmetry, $\Var( F^\bzero_t \mid A_x)= \Var( F^x_t \mid A_x)$ so
\begin{align}\label{n29.8}
\Var( F^\bzero_t \mid A_x) 
\leq  2/(1-\delta) .
\end{align}

Recall events $G_y$ and let $H_y = \bigcup_{1\leq v \leq y} (G_v \cap G_{-v})$ for $y\geq 1$.
By independence of $G_y$'s, $\P(H_y^c) = (1- e^{-2t})^ y$ for $y\geq 1$.
A part of the mass that was present at $-y $ and  $y$ at time 0 could have moved between vertices $\bzero$ and $1$ during the time interval $[0,t]$ only if $H_y$ failed. Hence,
\begin{align*}
a:= \E|F^\bzero_t| \leq \E M^\bzero_0 + \sum_{y \geq 1} \P(H_y^c) 
\E (M^y_0+ M^{-y}_0)
\leq 1 + \sum_{y \geq 1} 2 (1- e^{-2t})^ y < \infty .
\end{align*}

Recall that $\P(A_x) \geq 1-\delta >1/2$ for $x\geq x_1$.
 We obtain from \eqref{n29.4}  and \eqref{n29.7},
\begin{align*}
\E(\xi \bone_{A_x})^2
&=
\P(A_x) \Var( \xi \mid A_x) 
+ \left(1-\frac{  \P(A_x) -1}{ \P(A_x)}\right) \left(\E(\xi \bone_{A_x})\right)^2\\
&\leq \Var( \xi \mid A_x) 
+ 2 \left(\E(\xi \bone_{A_x})\right)^2.
\end{align*}
We apply this formula to $\xi = F^\bzero_t $ and use \eqref{n29.8} to see that for $x\geq x_1$,
\begin{align}\label{de18.5}
\E(F^\bzero_t \bone_{A_x})^2
&\leq
 \Var( F^\bzero_t \mid A_x) 
+ 2 \left(\E(F^\bzero_t \bone_{A_x})\right)^2
\leq 2/(1-\delta) + 2a^2.
\end{align}

 It is easy to see that $A_x^- \subset A_{x+2}^-$ for all odd $x > 6$ and $\P(A_x^-) \to 1$ as $x\to \infty$. Given $A_x^-$, the event $A_x^+$ and the random variable $F^\bzero_t$ are independent so the conditional distribution of $F^\bzero_t$ given $A_x^-$ is the same as the conditional distribution of $F^\bzero_t$ given $A_x$. This implies that
\begin{align}\label{de18.6}
&\E(F^\bzero_t \bone_{A_x^-})^2 = 
\E((F^\bzero_t \bone_{A_x^-})^2 \mid A_x^-) \P(A_x^-)
= 
\E((F^\bzero_t )^2 \mid A_x^-) \P(A_x^-)\\
&= 
\E((F^\bzero_t )^2 \mid A_x) \P(A_x^-)
= \frac{\P(A_x^-) }{\P(A_x) }
\E((F^\bzero_t \bone_{A_x})^2 \mid A_x) \P(A_x) 
= \frac{\P(A_x^-) }{\P(A_x) }\E(F^\bzero_t \bone_{A_x})^2.\nonumber
\end{align}
Since $\lim_{x\to \infty } \P(A_x) = \lim_{x\to \infty } \P(A_x^-)=1$, the last formula and \eqref{de18.5} imply that for some $x_2$ and all $x\geq x_2$,
$\E(F^\bzero_t \bone_{A_x^-})^2
\leq 2/(1-\delta) + 2a^2 + \delta$.
By Fatou's Lemma, $\E(F^\bzero_t )^2 \leq 2/(1-\delta) + 2a^2+\delta$, so
\begin{align}\label{n29.9}
\Var F^\bzero_t \leq 
\E(F^\bzero_t )^2
\leq 2/(1-\delta) + 2a^2+\delta.
\end{align}
Recall that $\E |F^\bzero_t| < \infty$. By symmetry, 
$\E F^\bzero_t=0$. By dominated convergence, $\lim_{x\to \infty} \E(F^\bzero_t \bone_{A_x^-}) = 0$. A calculation similar to that in \eqref{de18.6} shows that $\E(F^\bzero_t \bone_{A_x}) =\frac{\P(A_x) }{\P(A_x^-) } \E(F^\bzero_t \bone_{A_x^-})$.
This, the previous observation and the fact that $\lim_{x\to \infty } \P(A_x) = \lim_{x\to \infty } \P(A_x^-)=1$ imply that $\lim_{x\to \infty} \E(F^\bzero_t \bone_{A_x}) = 0$. Hence, we can strengthen \eqref{de18.5}
to see that for any $\delta >0$, some $x_3$ and all $x\geq x_3$,
\begin{align*}
\E(F^\bzero_t \bone_{A_x})^2
&\leq
 \Var( F^\bzero_t \mid A_x) 
+ 2 \left(\E(F^\bzero_t \bone_{A_x})\right)^2
\leq 2/(1-\delta) +\delta.
\end{align*}
This allows to strengthen \eqref{n29.9} as follows,
\begin{align*}
\Var F^\bzero_t \leq 
\E(F^\bzero_t )^2
\leq 2/(1-\delta) +2\delta.
\end{align*}
Since $\delta>0$ is arbitrarily small,
this completes the proof.
\end{proof}

We will now introduce an alternative representation of the meteor process on $\Z$. The mass at each vertex will be represented by ordered particles. Any two particles will be always ordered in the same way, no matter at which vertex they reside.
Let 
\begin{align*}
\Gamma^0_0 &= 0,\\
\Gamma^k_0 &= \sum_{j=0}^{k-1} M^j_0, \qquad k\geq 1,\\
\Gamma^k_0 &= -\sum_{j=k}^{-1} M^j_0, \qquad k\leq -1.
\end{align*}
We define $\Gamma^k_t$ to be a piecewise constant RCLL function with values in $\R$ as follows. The function $\Gamma^k_\cdot$ jumps at time $t$ only if $N^{k-1}_t = N^{k-1}_{t-} +1$ or $N^{k}_t = N^{k}_{t-} +1$. If $N^{k-1}_t = N^{k-1}_{t-} +1$ then $\Gamma^k_\cdot$ jumps at time $t$ to $(\Gamma^{k-1}_{t-} + \Gamma^k_{t-})/2$.
If $N^{k}_t = N^{k}_{t-} +1$ then $\Gamma^k_\cdot$ jumps at time $t$ to $(\Gamma^{k}_{t-} + \Gamma^{k+1}_{t-})/2$. 

Heuristically speaking, functions $\Gamma^k$ play a role similar to that of the cumulative distribution function for probability distributions on the real line. A further heuristic interpretation of these functions is that they determine the positions of infinitely many (uncountably many) particles moving along non-crossing trajectories in the following sense. The particle with label $y\in \R$ is located at vertex $k \in \Z$ at time $t$ if an only if $\Gamma^k_t \leq y < \Gamma^{k+1}_t$. We formalize this by defining $H^y_t$, the position of $y$ at time $t$, to be the unique $k$ such that $\Gamma^k_t \leq y < \Gamma^{k+1}_t$. Note that for all $x,y\in \R$ and $s,t\geq 0$, $(H^x_s - H^y_s)(H^x_t - H^y_t) \geq 0$, i.e., $x$ and $y$ are always ordered in the same way.

\begin{proposition}\label{de19.3}
Suppose that $\calM_0$ has the distribution $Q_\infty$. Then for every $\alpha \in (0,2)$ there exists $c_1 < \infty$ such that for all $y\in \R$ and $t\geq 0$, $\E(|H^y_t - H^y_0|^\alpha) < c_1$.
\end{proposition}

\begin{proof}
Recall the notation $F^x_t$ from Theorem \ref{n30.2}. Let $k_0 = H^y_0$.
Consider any integer $m\geq 1$ and suppose that $|H^y_t - H^y_0|> m$. By symmetry, it will suffice to analyze the case $H^y_t - H^y_0> m$. If this event occurred then 
at least one of the following events occurred,
\begin{align*}
A_1 &= \{F^{k_0} _t \geq m/2\},\\
A_2 &= \left\{ \sum_{j=k_0+1}^{k_0+m} M^j_t \leq m/2\right\}.
\end{align*}
By Theorem \ref{n30.2},
\begin{align}\label{de19.1}
\P(A_1) \leq \frac{\Var F^{k_0} _t}{(m/2)^2} \leq 8 m^{-2}.
\end{align}
Corollary \ref{n19.1} implies that
\begin{align*}
\P(A_2) \leq 
\P \left( \left|\sum_{j=k_0+1}^{k_0+m} M^j_t - m\right| \geq m/2\right)
\leq
\frac{\Var \sum_{j=k_0+1}^{k_0+m} M^j_t}{(m/2)^2} \leq 4 m^{-2}.
\end{align*}
This and \eqref{de19.1} imply that, for $m \geq 1$,
\begin{align}\label{de19.2}
\P(|H^y_t - H^y_0|> m) = 2 \P(H^y_t - H^y_0> m)
\leq 2 (\P(A_1) + \P(A_2) ) \leq 24 m^{-2}.
\end{align}
Note that the inequality also holds for real $m \geq 1$.
It is well known that for a non-negative random variable $\xi$ and $\alpha >0$, 
\begin{align*}
\E\xi^\alpha
=  \alpha\int_0^\infty a^{\alpha-1} \P(\xi>a) da.
\end{align*}
Hence, \eqref{de19.2} yields for every $\alpha \in(0,2)$ and some $c_1=c_1(\alpha)<\infty$,
\begin{align*}
\E(|H^y_t - H^y_0|^\alpha) &=  \alpha\int_0^\infty a^{\alpha-1} \P(|H^y_t - H^y_0|>a) da \\
&\leq \alpha\int_0^1 a^{\alpha-1}  da 
+ \alpha\int_1^\infty a^{\alpha-1} 24 a^{-2} da < c_1 .
\end{align*}
\end{proof}

\begin{remark}
(i) One may ask whether the condition $\alpha < 2$ in Proposition \ref{de19.3} is sharp. We believe that it is not. We conjecture that for every $\alpha < \infty$ there exists $c_1 < \infty$ such that for all $y\in \R$ and $t\geq 0$, $\E(|H^y_t - H^y_0|^\alpha) < c_1$. Similarly, we believe that the uniform bound in Theorem \ref{n30.2} can be extended to every moment of $F^\bzero_t$.

(ii) We present an informal but easy to formalize argument showing that on a circular graph, a particle in an ordered system of particles cannot move more than two full circles around the graph. Consider the circular graph $\calC_n$ and identify its vertices with points $e^{i2\pi k/n}$ on the unit circle $S$. Let $\wt H^\theta_0$ denote the position of the particle with label $\theta$ at time 0, where $e^{i\theta}\in S$. Processes $\wt H^\theta_t$ are defined in a way analogous to that for $H^y_t$; we leave the details of the construction to the reader. Let $a = \int_0^{2\pi} \wt H^\theta_0 d\theta$. Note that when a meteor hits at time $t$, half of the mass is moved $2\pi/n$ radians in the clockwise direction and half of the mass is moved $2\pi/n$ radians in the opposite direction. Hence, for every meteor hit time $t\geq 0$, $\int_0^{2\pi} \wt H^\theta_t d\theta = \int_0^{2\pi} \wt H^\theta_{t-} d\theta$. It follows that $\int_0^{2\pi} \wt H^\theta_t d\theta = a$ for all $t\geq 0$. Suppose that a particle moved more than two full circles around the graph, say, in the clockwise direction, between times 0 and $t$. Then, because all particles are ordered, all other particles must have moved at least one full circle in the clockwise direction. Thus $\int_0^{2\pi} \wt H^\theta_t d\theta \geq a + (2\pi)^2 > a$. This is a contradiction. It shows that a particle in an ordered system of particles can never move more than two full circles around the graph. A slight improvement of the argument shows that a particle can never move by the angle $2\pi(1+1/n)$.

(iii)
The representation of the meteor model using non-crossing functions $H^y$ is similar in spirit to some other models known in the literature. One of them is the motion of the tracer particle in the
exclusion process, see, e.g., \cite{Arr}. Another one is the trajectory of a particle in one of several models for reflecting paths proposed by Harris \cite{Harris} and Spitzer \cite{Spitzer}, and later generalized and carefully analyzed in \cite{DGL}. In all the cited models, the variance of the reflecting particle location grows with time (as a power of time). It is rather surprising that the variance of $H^y_t$ is not growing with time. This may be related to the fact that in the model of Howitt and Warren \cite{HW}, the mass redistribution function has to be rescaled as in 
\cite[(1.4)]{HW} for the limit in their theorems to be
non-degenerate.

(iv) An intriguing problem of ``number variance saturation'' was studied
in \cite{HJ}. At this point it is not clear whether the resemblance between that phenomenon and our Proposition \ref{de19.3}
is more than superficial.
\end{remark}

\section{Support of the stationary measure}\label{sec:supp}

Consider a connected simple graph $G$ with a finite vertex set $V$ and let $k=|V|$. Note that  $\calM_t \in [0,\infty)^V$. 
Recall from Section \ref{sec:pre} that there exists
a unique stationary measure $Q$ for the meteor process with $\sum _{x\in V} M^x_0 = k$. Let $U_Q$ be the closure of the support of $Q$ in $[0,\infty)^V$. We define $U$ to be the (closed) subset of 
$[0,\infty)^V$ which consists of all $\{a_x, x\in V\}$ such that $\sum_{x\in V} a_x = k$ and $a_x =0$ for at least one $x\in V$.

\begin{theorem}
We have $U_Q = U$.
\end{theorem}

\begin{proof}
The inclusion $U_Q \subset U$ is obvious. We will prove the opposite inclusion.

\medskip

{\it Step 1}. 
Recall that $d_x$ denotes the degree of vertex $x \in V$.
For $\ba = \{a_x, x\in V\} \in [0,\infty)^V$ and a vertex $y\in V$, we define $\calT(\ba, y) = \{b_x, x\in V\}$ by setting
$b_y = 0$, 
and $b_x = a_x + b_y/d_y$ for all $x \lra y$. We let $b_x = a_x$ if $x \not\lra y$ and $x\ne y$.
Note that the operation $\calT$ encodes the jump of the meteor process when a meteor hits vertex $y$.

We will now define an ``inverse'' operation to $\calT$. 
For a vertex $y\in V$ and  $\ba = \{a_x, x\in V\} \in [0,\infty)^V$, let $a_y^{\min} =\min_{x\lra y} a_x $. 
We let
$\calR(\ba, y) = \{b_x, x\in V\} $, where
$b_y = a_y + d_y a_y^{\min}$, 
and $b_x = a_x - a_y^{\min}$ for all $x \lra y$. We let $b_x = a_x$ if $x \not\lra y$ and $x\ne y$.
We will typically apply $\calR$ to $\ba $ and $y$ such that $\ba \in U$, $a_y =0$ and $a_y^{\min}  >0$. It is easy to see that if $\ba$ and $y$ satisfy these conditions then   
\begin{align}\label{ag13.1}
\calT(\calR(\ba,y),y) =\ba, \qquad \calR(\ba,y) \in U.
\end{align}

For $\ba, \bb \in [0,\infty)^V$, $\ba =\{a_x, x\in V\}$, $\bb =\{b_x, x\in V\}$, let $|\ba - \bb| = \sum_{x\in V} |a_x-b_x|$. In other words, $|\ba - \bb|$ is the $L^1$ norm of $\ba - \bb$. This implies that we have the usual triangle inequality $|\ba - \bc|\leq |\ba - \bb|+|\bb - \bc|$ for $\ba, \bb,\bc \in [0,\infty)^V$.

Let $U^*$ be the set of all
$\ba=\{a_x, x\in V\}\in U$ such that $a_x + a_y >0$ for all $x\lra y$. 
Fix any $\ba\in U^*$.
Let $a^1_{\min} $  be the minimum of non-zero $a_x$'s, and fix any $\eps_1\in(0, a^1_{\min}/2)$. We will define inductively infinite sequences of real numbers $\eps_1, \eps_2, \dots$, vertices $x_1, x_2, \dots$ and elements $\ba^1, \ba^2, \dots$ of $U$. 

We let $\ba^1 = \{a^1_x, x\in V\} =\{a_x, x\in V\}=\ba$. Let $z_1, \dots, z_m$ be all vertices such that $a^1_{z_r} = 0$ for $r=1,\dots, m$. We find $\delta >0$ so small that $\delta < \eps_1 / 2$ and we find $y$ such that $\delta < a^1_y$. We define $\wt \ba^1 = \{\wt a^1_x, x\in V\} \in U $ by setting $\wt a^1_{z_r} = \delta/2^r$ for $r=2,\dots, m$, $\wt a^1 _y = a^1_y - \sum_{r=2}^m \delta/2^r$, and $\wt a^1_x = a^1_x$ for all $x \ne y,z_2, \dots , z_m$.
Note that $\wt a^1_x=0$ if and only if $x= z_1$.
We let $\ba^2 = \calR(\wt \ba^1, z_1)$ and $x_1 = z_1$.
We let $a^2_{\min} $  be the minimum of non-zero $a^2_x$'s, and choose $\eps_2\in(0, (\eps_1 \land a^2_{\min})/2)$.

For the induction step, we assume that $\eps_1, \dots, \eps_j$, $x_1, \dots, x_{j-1}$ and $\ba^1, \dots, \ba^j$ have been defined for some integer $j\geq 2$.
Write $\ba^j = \{a^j_x, x\in V\}$ and suppose that $z^j_1, \dots, z^j_{m_j}$ are all vertices with $a^j_{z^j_r} = 0$ for $r=1,\dots, m_j$. 
For $x\in \{z^j_1, \dots, z^j_{m_j}\}$,
let $\alpha(x)$ be the smallest $\ell \leq j$ such that $x$ belongs to every sequence $z^r_1, \dots, z^r_{m_r}$ for $r=\ell, \dots, j$.
We can and will assume that the sequence $z^j_1, \dots, z^j_{m_j}$ is ordered in such a way that $\alpha(z^j_1) \leq \alpha(z^j_r)$ for all $r=2, \dots, m_j$.
We find $\delta_j >0$ so small that $\delta_j < \eps_j / 2^j$ and we find $v$ such that $\delta_j < a^j_v$. We define $\wt \ba^j =\{\wt a^j_x, x\in V\}  \in U $ by setting $\wt a^j_{z^j_r} = \delta_j/2^r$ for $r=2,\dots, m_j$, $\wt a^j _v = a^j_v - \sum_{r=2}^m \delta_j/2^r$, and $\wt a^j_x = a^j_x$ for all $x \ne v, z^j_2,  \dots , z^j_{m_j}$.
Note that
\begin{align}\label{ag15.2}
|\wt \ba^j - \ba^j| = 2\sum_{r=2}^m \delta_j/2^r \leq \delta_j < \eps_j/2^j.
\end{align}
We have $\wt a^j_x=0$ if and only if $x= z^j_1$.
We let $\ba^{j+1} = \calR(\wt \ba^j, z^j_1)$ and $x_j = z^j_1$. We let $a^{j+1}_{\min} $  be the minimum of non-zero $a^{j+1}_x$'s, and choose $\eps_{j+1}\in(0, (\eps_j \land a^{j+1}_{\min})/2)$.
This completes
the inductive definition of sequences $\eps_1, \eps_2, \dots$, $x_1, x_2, \dots$ and $\ba^1, \ba^2, \dots$ 

Note that $a^{j+1}_{z^j_1} >0$ and recall how we have 
used the function $\alpha(\,\cdot\,)$ to
choose an element of $z^j_1, \dots, z^j_{m_j}$ to be in the first position, i.e., $z^j_1$. It follows easily that for every vertex $x\in V$, there exist infinitely many $j$ such that $a^j_x >0$.

In view of \eqref{ag13.1}, we have
\begin{align}\label{ag15.3}
\calT(\ba^n, x_{n-1}) = \wt \ba^{n-1}, \qquad n\geq 2.
\end{align}

\medskip

{\it Step 2}. 
Consider an integer $n_0\geq 1$. Let $\bb^{n_0} = \ba^{n_0}$ and define $\bb^n$ for $n_0-1, n_0-2, \dots , 1$ by $\bb^n = \calT(\bb^{n+1}, x_{n}) $. 
We will show that
\begin{align}\label{ag15.1}
|\bb^n - \ba^n| \leq \sum_{m=n}^{n_0-1} \eps_n/2^m\leq \eps_n,
\qquad n=1, \dots, n_0.
\end{align}
By the definition of $\bb^{n_0}$, the estimate holds for $n=n_0$. We will prove the formula for other $n$ by induction. Suppose that the formula holds for some $n\in[2,n_0]$. We will show that it holds for $n-1$.

It follows from the proof of Theorem 3.2 in \cite{BBPS} (see the first displayed formula in that proof) that 
for any $\bc^1, \bc^2 \in [0,\infty)^k$ and $x\in V$,
\begin{align}\label{ag15.4}
|\calT(\bc^1, x) - \calT(\bc^2, x)|
\leq |\bc^1 - \bc^2|.
\end{align}

We have by the definition of $\bb^n$, \eqref{ag15.3}, \eqref{ag15.4}, induction step assumption \eqref{ag15.1} and \eqref{ag15.2},
\begin{align*}
|\bb^{n-1} - \ba^{n-1}| &= 
|\calT(\bb^{n}, x_{n-1}) - \ba^{n-1}|\\
&\leq |\calT(\bb^{n}, x_{n-1}) - \wt\ba^{n-1}|  + |\wt\ba^{n-1} - \ba^{n-1}| \\
&= |\calT(\bb^{n}, x_{n-1}) - \calT(\ba^{n}, x_{n-1})|  + |\wt\ba^{n-1} - \ba^{n-1}| \\
& \leq |\bb^{n} - \ba^{n}|  + |\wt\ba^{n-1} - \ba^{n-1}|\\
&\leq \sum_{m=n}^{n_0-1} \eps_n/2^m + \eps_{n-1}/2^{n-1}\\
&\leq \sum_{m=n}^{n_0-1} \eps_{n-1}/2^m + \eps_{n-1}/2^{n-1} \\
&= \sum_{m=n-1}^{n_0-1} \eps_{n-1}/2^m.
\end{align*}
This completes the induction step and thus completes the proof of 
\eqref{ag15.1}.

\medskip

{\it Step 3}. 
We will next prove that for every $y\in V$, the sequence $x_1, x_2, \dots$ contains
infinitely many $y$'s. Suppose otherwise. Let $V_1$ be the set of all $y\in V$ such that the sequence $x_1, x_2, \dots$ contains
infinitely many $y$'s and let $k_1 = |V_1|$. By assumption, $k_1 < k$. Since $V$ is finite, $k_1 >0$. 

Recall that $\wt a^j_x=0$ if and only if $x= z^j_1$,
$\ba^{j+1} = \calR(\wt \ba^j, z^j_1)$ and $x_j = z^j_1$. It follows that $a^{j+1}_x=0$ only if (but not necessarily if) $x\lra x_j$. In particular, if $a^{j+1}_x=0$ then $x\ne x_j$.
This implies that $k_1 \geq 2$. 
Another consequence of the fact that $a^{j+1}_x=0$ only if $x\lra x_j$ is that
$V_1$ is a connected subset of $V$. 

By assumption, $V_1^c := V\setminus V_1 \ne \emptyset$. 
Let $n_1$ be so large that $x_j \in V_1$ for all $j\geq n_1$.
We have noted earlier in the proof that
for every vertex $x\in V$, there exist infinitely many $j$ such that $a^j_x >0$.
Let $n_2 \geq n_1$ be such that for some $y\in V_1$, we have $a^{n_2}_y >0$.
By the definition of $\eps_{n_2}$, there exists $y\in V_1$ such that $a^{n_2}_y > 2 \eps_{n_2}$.
It follows from \eqref{ag15.1} applied with $n=n_2$ that for any $n_0>n_2$ and $\bb^{n_0} = \ba^{n_0}$, there exists $y\in V_1$ with $b^{n_2}_y > \eps_{n_2}$. 

Let $\{\wt X _n, n\geq 1\}$ be a discrete symmetric random walk on $V$ and let $\{X _n, n\geq 1\}$ be the process $X $ killed upon exiting $V_1$. Let $p_n(x,y)$ be the $n$-step transition probabilities for $ X $. Since $V_1^c \ne \emptyset$ and $V$ is connected, it follows that no matter what $X _0$ is, the process $ X $ will be killed at a finite (random) time, a.s., and, therefore, for any fixed $x,y\in V_1$, $\lim_{n\to \infty} p_n(x,y) = 0$. Since $V_1$ is a finite set, 
$\lim_{n\to \infty} \sup_{x,y \in V_1} p_n(x,y) = 0$. We choose $n_3$ so large that 
\begin{align}\label{ag22.1}
\sup_{n\geq n_3} \sup_{x \in V_1}  \sum_{y \in V_1} p_n(x,y)  \leq \eps_{n_2}/(2k).
\end{align}

We will say that $y_1, y_2, \dots, y_n$ is a nearest neighbor path in $V_1$ if $y_j \in V_1$ for all $j$ and $y_j \lra y_{j+1}$ for all $1\leq j \leq n-1$.
It is easy to see that we can choose $n_0$ so large that the following holds

(A1) For each nearest neighbor path $y_1, y_2, \dots, y_{n_3}$ of length $n_3$ in $V_1$, there exist $j_1, j_2, \dots , j_{n_3}$ such that $n_0>j_1 > j_2 > \dots > j_{n_3}>n_2$ and $x_{j_m} = y_m$ for $1 \leq m \leq n_3$.

Let 
\begin{align*}
\gamma &= \sum_{x\in V_1} b^{n_0}_x, \\
\bp_x &= b^{n_0}_x/\gamma, 
\qquad x\in V_1,
\end{align*}
and note that $\{\bp_x, x\in V_1\}$ is a probability distribution on $V_1$.

Suppose that the initial distribution of $X $ is given by $\P(X _1 =x) =\bp_x$ for $x\in V_1$. We will define a process $\{Y_n, 1\leq n\leq n_0-n_2+1\}$ which, heuristically speaking, represents the process $X $ slowed down so that it is
moving to the next step along its trajectory only when the current $x_j$ agrees with its location. The rigorous definition is the following. We set $Y_1 = X _1$ and $k_1 = 1$. Suppose that $Y_n$ and $k_n$ have been defined for some $n < n_0-n_2+1$. If $Y_n = x_{n_0-n} $ then we let $k_{n+1} = k_n +1$ and $Y_{n+1} = X _{k_{n+1}}$. Otherwise we let $Y_{n+1} = Y_n$ and $k_{n+1} = k_n$.

Let $\zeta$ be the time $n\leq n_0-n_2+1$ when $Y_n$ is killed (upon exiting $V_1$); we let $\zeta = n_0-n_2+2$ if there is no such time. It follows from (A1) that $Y_n$ makes at least 
$n_3$ steps on the interval $[1,n_0-n_2+1]$ or it is killed at $\zeta \leq n_0-n_2+1$. Hence, in view of \eqref{ag22.1}, $\P(\zeta > n_0-n_2+1) \leq \eps_{n_2}/(2k)$. 
It is elementary to check that $\P(Y_n = x) = b^{n_0 -n +1}_x/\gamma$ for $x\in V_1$ and $1\leq n \leq n_0-n_2+1$. In particular, 
$\P(Y_{n_0-n_2+1} = x) = b^{n_2}_x/\gamma $ for $x\in V_1$.
We obtain, using \eqref{ag22.1}, for all $x\in V_1$,
\begin{align*}
b^{n_2}_x &= \gamma \P(Y_{n_0-n_2+1} = x)
\leq k \P(Y_{n_0-n_2+1} = x) \leq 
k \P(\zeta > n_0-n_2+1) \leq k \eps_{n_2}/(2k)\\
& = \eps_{n_2}/2.
\end{align*}
This contradicts the fact that
there exists $y\in V_1$ with $b^{n_2}_y > \eps_{n_2}$.
This completes the proof 
that for every $y\in V$, the sequence $x_1, x_2, \dots$ contains
infinitely many $y$'s.

\medskip

{\it Step 4}. 
Recall the following: (i) We fixed an arbitrary $\ba\in U^*$; (ii) $a^1_{\min} $ is the minimum of non-zero $a_x$'s; (iii) $\eps_1\in(0, a^1_{\min}/2)$ is a fixed, arbitrarily small number; (iv) the sequence $x_1, x_2, \dots$ was constructed from $\ba$.
For a given $\bc \in U$ and $n$, we let $\bc^{n} = \bc$ and we define $\bc^j$ for $j=n-1, n-2, \dots , 1$ by $\bc^j = \calT(\bc^{j+1}, x_{j}) $.
We will show that there exists 
$n_4$ so large that for all $n\geq n_4$ and all $\bc \in U$,
\begin{align}\label{ag28.1}
|\bc^1 - \ba| \leq  2\eps_1.
\end{align}

Let $\{X^1_i, i\geq 1\}$ and $\{ X^2_i, i\geq 1\}$ be independent discrete time symmetric random walks on $G$.
Their initial distributions will be specified below.

Consider a large $n$ whose value will be specified later.
We define random walks $\{Y^1_i, 1\leq i\leq n\}$ and $\{ Y^2_i, 1 \leq  i\leq n\}$ with ``time delay'' as follows. For $m=1,2$,
we let $Y^m_1 = X^m_1$ and $\beta^m_1 = 1$. Consider any $2\leq j\leq n$ and suppose that $Y^m_{j-1}$ and $\beta^m_{j-1}$ have been defined. Let
\begin{align*}
\beta^m_j &= 
\begin{cases}\beta^m_{j-1}+1 & \text{if  } Y^m_{j-1} = x_{n-j+1},\\
\beta^m_{j-1} & \text{otherwise},
\end{cases} \\
Y^m_j &= X^m_{\beta^m_j}.
\end{align*}
In other words, $Y^m$ visits the same vertices as $X^m$ does, in the same order, but it changes the location between times $j-1$ and $j$ if and only if $Y^m_{j-1} = x_{n-j+1}$.

Let $d_G = \max\{d_x: x\in V\}$ be the degree of the graph $G$. Let $\dist(x,y)$ be the graph distance between $x,y\in V$. For $0 \leq j \leq n-1$,
\begin{align}\label{s2.1}
\P\left(\dist(Y^1_{j+1},Y^2_{j+1}) = \dist(Y^1_{j},Y^2_{j})-1
\mid Y^1_j \ne Y^2_j, (Y^1_{j+1},Y^2_{j+1}) \ne (Y^1_j, Y^2_j) \right)
\geq 1/d_G.
\end{align}
Let $\tau =\min\{j: 1\leq j \leq n, Y^1_j =  Y^2_j \}$ with the convention that $\min\emptyset = \infty$. 
We obtain from \eqref{s2.1}, for $\ell \leq n-1$, and any $x,y\in V$,
\begin{align*}
\P\left(\tau \leq \ell
\mid \beta^1_{\ell-1} + \beta^2_{\ell-1} \geq k+1, Y^1_1 =x, Y^2_1=y \right)
\geq 1/d_G^k.
\end{align*}
Therefore,
\begin{align*}
\P\left(\tau > \ell
\mid \beta^1_{\ell-1} + \beta^2_{\ell-1} \geq k+1, Y^1_1 =x, Y^2_1=y \right)
\leq 1- 1/d_G^k.
\end{align*}
This and the Markov property imply that
\begin{align}\label{s2.2}
\P\left(\tau = \infty
\mid \beta^1_{n-1} + \beta^2_{n-1} \geq (k+1) m, Y^1_1 =x, Y^2_1=y \right)
\leq (1- 1/d_G^k)^m.
\end{align}
We now fix $m_0$ such that $(1-1/d_G^{k})^{m_0} < \eps_1/(2k)$.

Recall that for every $y\in V$, the sequence $x_1, x_2, \dots$ contains
infinitely many $y$'s.
This implies that there exists $n_4$ so large that   
for any sequence $y_1, y_2, \dots, y_{km_0}$ of elements of $V$ of length $km_0$, there exists a subsequence $x_{j_1}, x_{j_2}, \dots, x_{j_{km_0}}$ of 
$x_1, x_2, \dots, x_{n_4}$ such that $x_{j_m} = y_m$ for all $1\leq m \leq km_0$.
Recall the integer $n$ used in the definition of $Y^m$'s and assume that $n\geq n_4$.
It follows from the definition of $n_4$ that 
$\beta^1_{n-1} + \beta^2_{n-1} \geq (k+1) m_0$ with probability 1. 
Hence, with this choice of $n$, \eqref{s2.2} implies that for any $x,y \in V$,
\begin{align}\label{s2.3}
\P\left(\tau = \infty
\mid  Y^1_1 =x, Y^2_1=y \right)
\leq (1- 1/d_G^k)^{m_0} < \eps_1/(2k).
\end{align}

Recall that the sequence $\bc^j$ for $j=n,n-1,  \dots , 1$ was defined relative to $n$.
We let $\bb^n = \ba^n$ and we define $\bb^j$ for $j=n-1, n-2, \dots , 1$ by $\bb^j = \calT(\bb^{j+1}, x_{j}) $.
Let $\bb^j = \{b^j_x, x\in V\}$ and $\bc^j = \{c^j_x, x\in V\}$ for $1\leq j \leq n$, and $p^j_x = b^j_x/k$ and $q^j_x = c^n_x/k$ for $x\in V$. 
Note that 
$\bp^j:=\{p^j_x, x\in V\}$ and $\bq^j:=\{q^j_x, x\in V\}$ are probability distributions on $V$, for all $j$. 

Let $Y^3_j = Y^2_j$ for $j \leq \tau$ and $Y^3_j = Y^1_j$ for $j > \tau$.
Assume that the (initial) distribution of $Y^1_1$ is $\bp^n$ and the distribution of $Y^3_1= Y^2_1$ is $\bq^n$. 
It is easy to see that the distribution of $Y^1_j$ is $\bp^{n-j+1}$ and that of $ Y^3_j$ is the same as that of $Y^2_j$, and it is equal to $\bq^{n-j+1}$, for all $1\leq j\leq n$.
It follows from \eqref{s2.3} that 
\begin{align*}
|\bp^1 - \bq^1 | = \sum_{x\in V} |p^1_x - q^1_x|
\leq 2\P(Y^1_n \ne Y^3_n)
\leq \eps_1/k.
\end{align*}
Hence, $|\bb^1 - \bc^1| \leq \eps_1$.
This and \eqref{ag15.1} applied with $n=1$ show that 
$|\ba^1 - \bc^1| \leq 2\eps_1$. Thus, \eqref{ag28.1} is proved.

\medskip

{\it Step 5}. 
Consider any  time $t\geq 0$ and suppose that the first $n_4$ meteor hits after time $t$ occur at the sites $x_{n_4}, x_{n_4-1}, \dots, x_1$, in this order, and the last hit occurs at time $s >t$. It follows from \eqref{ag28.1} that $|\calM_s - \ba| \leq 2\eps_1$. A standard argument shows that, with probability 1, there exists $t\geq 0$ such that the first $n_4$ meteor hits after time $t$ occur at the sites $x_{n_4}, x_{n_4-1}, \dots, x_1$. Hence, $\calM$ will come within distance $2\eps_1$ of $\ba$ at some time, a.s., for any initial distribution of $\calM_0$.
Since $\ba$ is an arbitrary element of $U^*$, $\eps_1$ is an arbitrarily small positive number and $U_Q$ is closed (by definition), we have $U^* \subset U_Q$. It is easy to see that $U^*$ is dense in $U$. We conclude that $U\subset U_Q$, thus finishing the proof.
\end{proof}

\section{Acknowledgments}

I would like to thank Daniel Lanoue, Yves Le Jan and Anna Talarczyk-Noble for very useful advice.
I am grateful to the referees for very careful reading
of the original manuscript and many suggestions for improvement.

\bibliographystyle{plain}
\bibliography{meteor}

\end{document}